\documentclass[10pt]{article}
\RequirePackage{etex}
\font\smallit=cmti10

\usepackage[english]{babel}
\usepackage[utf8]{inputenc}
\usepackage{graphicx,changepage,xcolor, float, array, amssymb, amsfonts, mathrsfs, amsmath,comment, mathtools, enumerate, hyperref, url, amsthm, tikz-cd, latexsym, epsfig, xcolor}
\makeatletter

\renewcommand\section{\@startsection {section}{1}{\z@}
{-30pt \@plus -1ex \@minus -.2ex}
{2.3ex \@plus.2ex}
{\normalfont\normalsize\bfseries\boldmath}}

\newtheorem{manualtheoreminner}{Theorem}
\newenvironment{manualtheorem}[1]{%
  \renewcommand\themanualtheoreminner{#1}%
  \manualtheoreminner
}{\endmanualtheoreminner}

\newtheorem{manualdefinitioninner}{Definition}
\newenvironment{manualdefinition}[1]{%
  \renewcommand\themanualdefinitioninner{#1}%
  \manualdefinitioninner
}{\endmanualdefinitioninner}

\renewcommand\subsection{\@startsection{subsection}{2}{\z@}
{-3.25ex\@plus -1ex \@minus -.2ex}
{1.5ex \@plus .2ex}
{\normalfont\normalsize\bfseries\boldmath}}

\renewcommand{\@seccntformat}[1]{\csname the#1\endcsname. }

\makeatother

\newcommand\numberthis{\addtocounter{equation}{1}\tag{\theequation}}

\makeatletter
\newcommand*{\rom}[1]{\expandafter\@slowromancap\romannumeral #1@}
\makeatother

\makeatletter
\renewcommand*\env@matrix[1][*\c@MaxMatrixCols c]{%
	\hskip -\arraycolsep
	\let\@ifnextchar\new@ifnextchar
	\array{#1}}
\makeatother

\newtheorem{theorem}{Theorem}
\newtheorem{lemma}{Lemma}

\newtheorem{corollary}{Corollary}

\theoremstyle{definition}
\newtheorem{definition}{Definition}
\newtheorem{conjecture}{Conjecture}
\newtheorem{remark}{Remark}
\newtheorem{question}{Question}
\newtheorem{example}{Example}

\usepackage{setspace}
\DeclareMathOperator{\Gal}{Gal}

\DeclareMathOperator{\Frob}{Frob}

\newcommand{\Ocal}{\mathcal{O}}

\newcommand{\ZZ}{\mathbb{Z}}
\newcommand{\QQ}{\mathbb{Q}}

\newcommand{\NN}{\mathbb{N}}

\newcommand{\frakp}{\mathfrak{p}}
\newcommand{\frakq}{\mathfrak{q}}

\usepackage{autonum}

\begin{document}

\begin{center}
\uppercase{\bf On The Partition Regularity of }${\bf ax+by = cw^mz^n}$
\vskip 20pt
{\bf Sohail Farhangi}\\
{\smallit Department of Mathematics, the Ohio State University, Columbus, Ohio, United States of America}\\
{\tt sohail.farhangi@gmail.com}
\vskip 10pt
{\bf Richard Magner}\\
{\smallit Department of Mathematics, Boston University, Boston, Massachusetts, United States of America}\\
{\tt rmagner@bu.edu}
\end{center}
\vskip 20pt
\centerline{\smallit published version is available at http://math.colgate.edu/~integers/x18/x18.pdf } 
\vskip 30pt

\centerline{\bf Abstract}
\noindent
	We give a partial classification of the $a,b,c \in \mathbb{Z}\setminus\{0\}$ and $m,n \in \mathbb{N}$ for which the equation $ax+by = cw^mz^n$ is partition regular (PR) over $\mathbb{Z}\setminus\{0\}$. We show that if $m,n \ge 2$, then $ax+by = cw^mz^n$ is PR over $\mathbb{Z}\setminus\{0\}$ if and only if $a+b = 0$. Next, we show that if $n$ is odd, then the equation $ax+by = cwz^n$ is PR over $\mathbb{Z}\setminus\{0\}$ if and only if one of $\frac{a}{c}, \frac{b}{c},$ or $\frac{a+b}{c}$ is an $n$th power in $\mathbb{Q}$, and we come close to a similar characterization for even $n$. We also obtain some results about when a system of equations of the form $a_ix_i+b_iy_i = c_iw_iz_i^n, 1 \le i \le k$, is PR over an integral domain $R$. In order to show that the equation $ax+by = cwz^n$ is not PR over $\mathbb{Z}\setminus\{0\}$ for certain values of $a,b,c,$ and $n$, we prove a partial generalization of the criteria of Grunwald and Wang for when $\alpha \in \mathbb{Z}$ is an $n$th power modulo every prime $p$. In particular, we show that for any odd $n$ and any $\alpha,\beta,\gamma \in \mathbb{Q}$ that are not $n$th powers, there exist infinitely many primes $p \in \mathbb{N}$ for which none of $\alpha,\beta,$ and $\gamma$ are $n$th powers modulo $p$, and we obtain a similar but weaker result for even $n$. In order to show that for an integral domain $R$ the equation $ax+by = cwz^n$ is PR over $R\setminus\{0\}$ for certain values of $a,b,c$ and $n$, we use ultrafilters and the algebra of the Stone-\u{C}ech Compactification. This necessitates us to prove some new results of independent interest about ultrafilters $v \in \beta R$ for which every $A \in v$ is additively and multiplicatively central.
\pagestyle{myheadings}
\thispagestyle{empty}
\baselineskip=12.875pt
\vskip 30pt

\section{Introduction}

We say that an equation is \emph{partition regular} over a set $S$ if for any finite partition $S = \bigcup_{i = 1}^rC_i$, there exists some $C_i$ containing a solution to the equation. One of the first results about partition regular diophantine equations is the celebrated theorem of Schur \cite{SchurThm}, which established the partition regularity of $x+y = z$ over $\mathbb{N}$. Schur's student Rado \cite{RadosTheorem} classified which finite systems of linear homogeneous equations are partition regular over $\mathbb{N}$. Since it is not known whether or not the equation $x^2+y^2 = z^2$ is partition regular over $\mathbb{N}$, we are still far from achieving a classification for which systems of polynomial equations are partition regular. One of the first results in this direction is a theorem of Bergelson \cite[page 53]{ERTAnUpdate}, which shows that the equation  $x-y = p(z)$ is partition regular over $\mathbb{N}$ when $p(z)$ is a polynomial with integer coefficients which satisfies $p(0) = 0$. While the result of Bergelson shows that the equation $x-y = z^2$ is partition regular over $\mathbb{N}$, Csikv\'ari, Gyarmati and S\'arkozy \cite{a+b=cd} showed that the equation $x+y = z^2$ is \emph{not} partition regular over $\mathbb{N}$ (see also \cite{GreenLindqvist}), and asked whether the equation $x+y = wz$ is partition regular over $\mathbb{N}$. Their question was answered in the positive, independently by Bergelson \cite[Section 6]{BergelsonSurvey} and Hindman \cite{SumsEqualProducts}. Both proofs make use of ultrafilters and the algebra of the Stone-\v{C}ech compactifation. More examples of polynomial equations whose partition regularity is known can be found in \cite{NonlinearRado}, \cite{RamseyWarring}, \cite{NonStandardRamseyTheory}, \cite{JoelSumsAndProducts}, and \cite{PrendivilleDiagonalRamsey}.

The polynomial equations mentioned so far have a simple form, but the proofs of their partition regularity properties are quite specific to these cases. 
Since there is currently no unified theory for the partition regularity of polynomial equations, any class of equations whose partition regularity is known may provide insight towards such a theory. In this paper, we give a partial classification of the partition regularity of equations of the form $ax + by = cw^m z^n$, where $a,b,c \in \mathbb{Z} \setminus \{0\}$ and $m,n \in \mathbb{N}$ are parameters, and $x,y,z,w$ are variables. Theorem \ref{MainResult} is the main result of this paper. Before stating Theorem \ref{MainResult}, we note that we remove $0$ when considering partition regularity over a ring $R$ in order to avoid trivial solutions. We also recall that any equation which is partition regular over a set $S_1$ (such as $\mathbb{N}$) is also partition regular over any set $S_2$ that contains $S_1$ (such as $\mathbb{Z}\setminus\{0\})$, but the converse is not true in general.

\begin{theorem} \label{MainResult}
	Fix $a,b,c\in \mathbb{Z}\setminus\{0\}$ and $m,n \in \mathbb{N}$.
	\begin{itemize}
		\item[(i)] Suppose that $m,n \ge 2$.
		
		\begin{itemize}
		    \item[(a)] If $a+b \neq 0$, then the equation
		
		\begin{equation}	\label{MainEquation}
			ax+by = cw^mz^n
		\end{equation}
		
		is not partition regular over $\mathbb{Z}\setminus\{0\}$.
		
		\item[(b)] If $a+b = 0$, then Equation \eqref{MainEquation} is partition regular over $\mathbb{N}$.
		\end{itemize}
		\item[(ii)] If one of $\frac{a}{c}, \frac{b}{c},$ or $\frac{a+b}{c}$ is an $n$th power in $\mathbb{Q}$, then the equation 
		
		\begin{equation} 	\label{MainEquation2}
			ax+by = cwz^n
		\end{equation}
		is partition regular over $\mathbb{Z}\setminus\{0\}$. If one of $\frac{a}{c},\frac{b}{c},$ or $\frac{a+b}{c}$ is an $n$th power in $\mathbb{Q}_{\ge 0}$, then Equation \eqref{MainEquation2} is partition regular over $\mathbb{N}$.
		
		\item[(iii)] Assume that Equation \eqref{MainEquation2} is partition regular over $\mathbb{Q}\setminus\{0\}$.\footnote{To give a converse to part (ii) of Theorem \ref{MainResult} we would assume that Equation \eqref{MainEquation2} is partition regular over $\mathbb{N}$ or $\mathbb{Z}\setminus\{0\}$. The assumption that Equation \eqref{MainEquation2} is partition regular over $\mathbb{Q}\setminus\{0\}$ is weaker than either of the previous assumptions since $\mathbb{N} \subseteq \mathbb{Z}\setminus\{0\} \subseteq \mathbb{Q}\setminus\{0\}$.}
		
		\begin{enumerate}[(a)]
			\item If $n$ is odd, then one of $\frac{a}{c}, \frac{b}{c},$ or $\frac{a+b}{c}$ is an $n$th power in $\mathbb{Q}$.
			
			\item If $n \neq 4, 8$ is even, then one of $\frac{a}{c}, \frac{b}{c},$ or $\frac{a+b}{c}$ is an $\frac{n}{2}$th power in $\mathbb{Q}$.
			
			
			\item If $n$ is even, then either one of $\frac{a}{c}, \frac{b}{c},$ or $\frac{a+b}{c}$ is a square in $\mathbb{Q}$, or\\ $(\frac{a}{c})(\frac{b}{c})(\frac{a+b}{c})$ is a square in $\mathbb{Q}$.
			
		\end{enumerate}
	\end{itemize}
\end{theorem}

 To prove Theorem 1(i)(a) we use one of the Rado conditions for polynomial equations that is proven in \cite{NonlinearRado}, and we will show that Theorem 1(i)(b) is an easy consequence of the partition regularity of the equation $x-y = p(z)$. To prove Theorem 1(ii) we use ultrafilters similar to those that were used in \cite{BergelsonSurvey} and \cite{SumsEqualProducts}. To prove Theorem 1(iii) we use the classic partitions that were used by Rado in \cite{RadosTheorem} when determining necessary conditions for a finite system of linear homogeneous equations to be partition regular. In order to demonstrate that these classic partitions yield our desired results, we are required to prove Theorem \ref{ExistenceOfPrimesInIntro}, which is a partial generalization of the criteria of Grunwald and Wang for when $\alpha \in \mathbb{Z}$ is an $n$th power modulo every prime $p \in \mathbb{N}$ (see \cite{Grunwald}, \cite{Wang}, \cite{Wang-OnGrunwald}). In order to state Theorem \ref{ExistenceOfPrimesInIntro} we define the following notation. If $p \in \mathbb{N}$ is a prime and $r,s \in \mathbb{Z}$ are such that $p\nmid s$, then $\frac{r}{s}\equiv rs^{-1}\pmod{p}$.

\begin{theorem}\label{ExistenceOfPrimesInIntro}
	Let $\alpha, \beta, \gamma \in \mathbb{Q}\setminus\{0\}$. Suppose that either of the following hold: 
	\begin{enumerate}[(i)]
		\item $n$ is odd and $\alpha, \beta, \gamma $ are not $n$th powers; or
		\item $n$ is even, $\alpha, \beta, \gamma $ are not $\frac{n}{2}$th powers, and $\alpha$ is not an $\frac{n}{4}$th power if $4 \mid n$. 
	\end{enumerate}
	Then there exist infinitely many primes $p \in \NN$ modulo which none of $\alpha, \beta, \gamma $ are $n$th powers.
\end{theorem}
We remark that we will apply Theorem \ref{ExistenceOfPrimesInIntro} to $\alpha = \frac{a}{c}, \beta = \frac{b}{c},$ and $\gamma = \frac{a+b}{c}$. Since $\alpha = \beta + \gamma$, the condition that at least one of $\alpha, \beta,$ and $\gamma$ not be an $\frac{n}{4}$th power if $4\mid n$ is automatic by Fermat's Last Theorem when $n > 8$.

The structure of the paper is as follows. In Section \ref{SectionOnPreliminaries} we provide a statement of Rado's Theorem and briefly review some facts about the usage of ultrafilters in Ramsey Theory. A major goal of this section is to quickly familiarize the reader who is inexperienced with Ramsey Theory with enough basic knowledge of ultrafilters that they will be able to use as a blackbox the special kinds of ultrafilters introduced in Theorems \ref{SpecialUltrafilter} and \ref{SpecialUltrafilter6}. In Section \ref{PositiveResults} we prove items (i)(b), (ii), and (iii)(a-b) of Theorem \ref{MainResult}.

The main result of Section \ref{NumberTheorySection} is Theorem \ref{ExistenceOfPrimesInIntro}.
Since our proof of Theorem \ref{ExistenceOfPrimesInIntro} already requires us to work in finite extensions of $\mathbb{Q}$, we also prove a similar result as Lemma \ref{ExistenceOfPrimes} in the more general setting of rings of integers of number fields. We then prove item (iii)(c) of Theorem \ref{QuadraticCaseNegativeResult}, and we conclude the section with Lemma \ref{2VariableLemma} which is an analogue of Theorem \ref{ExistenceOfPrimesInIntro} for $2$ variables.

In Section \ref{ReductionToM==1} we prove of Theorem \ref{MainResult}(i)(a). We also determine the partition regularity of some equations of the form $ax+by = cwz^n$ which are not already addressed by Theorem \ref{MainResult} through the use of Lemma \ref{MinimalRadoConditionApplied}. In Section \ref{SectionForGeneralDomains} we investigate the partition regularity of $ax+by = cwz^n$ over general integral domains $R$ and attain results analogous to Theorem \ref{MainResult} while using methods very similar to those used in Section \ref{PositiveResults}. In Section \ref{SectionForSystems} we investigate the partition regularity of systems of equations of the form $a_ix_i+b_iy_i = c_iw_iz_i^n, 1 \le i \le k$ over integral domains $R$ and attain results that support Conjecture \ref{FinalConjectureIntro}.

\begin{conjecture}\label{FinalConjectureIntro}
    Let $a_1,\cdots,a_k,b_1,\cdots,b_k,c_1,\cdots,c_k \in \mathbb{Z}\setminus\{0\}$ and $n \in \mathbb{N}$. The system of equations

\begin{equation}
    \begin{array}{ccccc}
         a_1x_1 & + & b_1y_1 & = & c_1w_1z_1^n\\
         a_2x_2 & + & b_2y_2 & = & c_2w_2z_2^n\\
         & & \vdots & & \\
         a_kx_k & + & b_ky_k & = & c_kw_kz_k^n
    \end{array}
\end{equation}
is partition regular over $\mathbb{Z}\setminus\{0\}$ if and only if

\begin{equation}
    I := \bigcap_{i = 1}^k \left\{\frac{a_i}{c_i},\frac{b_i}{c_i},\frac{a_i+b_i}{c_i}\right\}
\end{equation}
contains an $n$th power in $\mathbb{Q}$.
\end{conjecture}

In Section \ref{Conclusion} we state some conjectures and examine some equations and systems of equations whose partition regularity remains unknown. We also elaborate on the distinction between partition regularity of a polynomial equation over $\mathbb{N}$ instead of $\mathbb{Z}\setminus\{0\}$ by considering some illustrative examples of polynomial equations over $\mathbb{Z}[\sqrt{2}]$.

The main purpose of Section \ref{Appendix} is to give a thorough proof of the existence of the ultrafilters in Theorems \ref{SpecialUltrafilter} and \ref{SpecialUltrafilter6} as Theorem \ref{SpecialUltrafilter2} and \ref{SpecialUltrafilter5}, respectively. To this end, we begin the section with a detailed introduction to the theory of ultrafilters and its applications to semigroup theory. While the ultrafilter that we use in Theorem \ref{SpecialUltrafilter2} is the same as the ultrafilter used in \cite{BergelsonSurvey} and \cite{SumsEqualProducts} (cf. Remark \ref{DiscussionOfDomains}), an analogous ultrafilter need not exist over a general integral domain $R$. Consequently, we prove Corollary \ref{SpecialUltrafilter4} and Theorem \ref{NecessityOfHomFin} to obtain a characterization of the integral domains $R$ which possess ultrafilters analogous to the one in Theorem \ref{SpecialUltrafilter2}.

\section{Preliminaries}
\label{SectionOnPreliminaries}
In this section we review some facts and useful tools in the study of partition regular equations.

\begin{definition}
	Given a set $S$, a ring $R$, and functions $f_1,\cdots,f_k:S^n\rightarrow R$, the system of equations
	
	\begin{equation}	\label{DefOfPR}
		\begin{array}{ccc}
			f_1(x_1,\cdots,x_n) & = & 0 \\
			f_2(x_1,\cdots,x_n) & = & 0 \\
			\vdots & \vdots & \vdots \\
			f_k(x_1,\cdots,x_n) & = & 0
		\end{array}
	\end{equation}
	is {\it partition regular over S} if for any finite partition $S = \bigcup_{i = 1}^rC_i$, there exists $1 \le i_0 \le r$ and $x_1,\cdots,x_n \in C_{i_0}$ which satisfy \eqref{DefOfPR}. If the set $S$ is understood from context, then we simply say that the system of equations is {\it partition regular}.
\end{definition}

\begin{definition}\label{ColumnsCondition}
	Let $R$ be an integral domain with field of fractions $K$. A matrix ${\bf M} \in \mathrm{M}_{m\times n}(R)$ satisfies the {\it columns condition} if there exists a partition $C_1,\cdots, C_k$ of the column indices such that for $\vec{s}_i = \sum_{j \in C_i} \vec{c}_j$ we have
	\begin{itemize}
		\item[(i)] $\vec{s}_1 = (0,\cdots,0)^T$;
		
		\item[(ii)] for all $i \ge 2$, we have
		
		\begin{equation}
			\vec{s}_i \in \text{Span}_{K}\{\vec{c}_j\ |\ j \in C_{\ell}, 1 \le \ell < i\}.
		\end{equation}
	\end{itemize}
\end{definition}

The columns condition was used by Rado to classify when a finite system of homogeneous linear equations is partition regular.
\begin{theorem}[\cite{RadosTheorem}] 		\label{RadosTheorem}
	Given ${\bf M} \in \mathrm{M}_{m\times n}(\mathbb{Z})$, the system of equations
	\begin{equation}  \label{RadosTheoremEquation}
		{\bf M}(x_1,\cdots,x_n)^T = 0
	\end{equation}
	is partition regular over $\mathbb{N}$ if and only if ${\bf M}$ satisfies the columns condition.
\end{theorem}

\begin{corollary}\label{1PREquation}
For $a_1,\cdots,a_s \in \mathbb{Z}$, the equation

\begin{equation}
    a_1x_1+\cdots+a_sx_s = 0
\end{equation}
is partition regular over $\mathbb{N}$ if and only if there exists $\emptyset \neq F \subseteq [1,s]$ for which $\sum_{i \in F}a_i = 0$.
\end{corollary}

Rado also characterized which finite, not necessarily homogeneous, linear systems of equations are partition regular.

\begin{theorem}[\cite{RadosTheorem}] \label{InhomogeneousRadoTheorem}
	Given ${\bf M} \in M_{m\times n}(\mathbb{Z})$ and $(b_1,\cdots,b_n) \in \mathbb{Z}^n$, the equation
	\begin{equation} \label{InhomogeneousRadoEquation}
		{\bf M}(x_1,\cdots,x_n)^T = (b_1,\cdots,b_n)^T
	\end{equation}
	is partition regular over $\mathbb{Z}$ if and only if Equation \eqref{InhomogeneousRadoEquation} admits an integral solution in which $x_1 = x_2 = \cdots = x_n$. 
\end{theorem}

We point out to the reader that in Theorem \ref{InhomogeneousRadoTheorem} it is possible to only obtain partition regularity in a trivial sense. For example, since $(x,y) = (0,0)$ is the only solution of the system of equations

\begin{equation}
    \begin{array}{ccrcc}
         x&-&y&=&0\textcolor{white}{,} \\
         x&-&2y&=&0,
    \end{array}
\end{equation}
we see that the system is not partition regular over $\mathbb{Z}\setminus\{0\}$ even though it is partition regular over $\mathbb{Z}$. As we will see in Section \ref{ReductionToM==1}, Corollary \ref{1PREquation} and Theorem \ref{InhomogeneousRadoTheorem} can be used in conjunction to determine whether or not a single linear equation is partition regular over $\mathbb{N}$ or $\mathbb{Z}$ in a nontrivial fashion.

\begin{theorem}[\cite{ERTAnUpdate}] \label{PolynomialSzemeredi}
	If $p(x) \in \mathbb{Z}[x]$ satisfies $p(0) = 0$, then the equation $x-y = p(z)$ is partition regular over $\mathbb{N}$.
\end{theorem}

The theory of ultrafilters has been very useful in the study of Ramsey theory and partition regular equations. We briefly recall some basic facts here and give a more detailed introduction in Section \ref{Appendix}. 

\begin{definition}\label{UltrafilterDefinition1}
	Given a set $S$ let $\mathcal{P}(S)$ be the power set of $S$. A collection of sets $p \subseteq \mathcal{P}(S)$ is an {\it ultrafilter} over $S$ if it satisfies the following properties:
	\begin{enumerate}[(i)]  
		\item The empty set is not in $p$.
		
		\item If $A \in p$ and $A \subseteq B$ then $B \in p$.
		
		\item If $A,B \in p$ then $A\cap B \in p$.
		
		\item For any $A \subseteq S$, either $A \in p$ or $A^c \in p$.
	\end{enumerate}
	We let $\beta S$ denote the space of all ultrafilters over $S$.
\end{definition}

It is often useful to think about $\beta S$ as the set of finitely additive $\{0,1\}$-valued measures on the set $S$. For now, we only require the following facts about ultrafilters. First, we see that for any finite partition of $S = \bigcup_{i = 1}^rC_i$ and any ultrafilter $p$, there exists $1 \le i_0 \le r$ for which $C_i \in p$ if and only if $i = i_0$. In fact, for any $A \in p$ and any finite partition $A = \bigcup_{i = 1}^rC_i$, there exists $1 \le i_0 \le r$ for which $C_i \in p$ if and only if $i = i_0$. Secondly, we note that if $p$ is an ultrafilter, $A \in p$ and $B \notin p$, then $B^c \in p$, so $A\setminus B = A\cap B^c \in p$. Lastly, we require the existence of a special kind of ultrafilter.

\begin{theorem}[cf. Theorem \ref{SpecialUltrafilter2}]\label{SpecialUltrafilter}
	There exists an ultrafilter $p \in \beta\mathbb{N}$ with the following properties.
	\begin{enumerate}[(i)]
		\item For any $A \in p$ and $\ell \in \mathbb{N}$, there exists $b,g \in A$ with $\{bg^j\}_{j = 0}^{\ell} \subseteq A$.
		
		\item For any $A \in p$ and $h,\ell \in \mathbb{N}$, there exists $a,d \in \mathbb{N}$ for which $\{hd\}\cup\{ha+id\}_{i = -\ell}^{\ell} \subseteq A$.
		
		\item For any $s \in \mathbb{N}$, we have $s\mathbb{N} \in p$. 
	\end{enumerate}
\end{theorem}

The proof of Theorem \ref{SpecialUltrafilter} requires more technical knowledge about ultrafilters. Since this technical knowledge is not needed in Sections \ref{SectionOnPreliminaries}-\ref{Conclusion}, we defer the proof of Theorem \ref{SpecialUltrafilter} to Section \ref{Appendix}.
\section{On the Partition Regularity of \texorpdfstring{$ax+by = cwz^n$}{ax+by=cwzn} over \texorpdfstring{$\mathbb{N}$}{N} and \texorpdfstring{$\mathbb{Z}\setminus\{0\}$}{Z0}}
\label{PositiveResults}
The purpose of this section is to Theorems 1(i)(b), 1(ii), 1(iii)(a), and 1(iii)(b). We begin by proving Theorem \ref{MainResult}(i)(b) of since it is an easy consequence of results in the existing literature. For $s \in \mathbb{Z}$ and $A \subseteq \mathbb{Z}$ we write $sA := \{sa\ |\ a \in A\}$ and $A/s:= \{z \in \mathbb{Z}\ |\ sz \in A\}$.

\begin{lemma}\label{ScaledBergelsonLemma}
    If $a,s \in \mathbb{N}$ and $p(x) \in \mathbb{Z}[x]$ satisfies $p(0) = 0$, then the equation
    
    \begin{equation}
        ax-ay = z^{n+1}
    \end{equation}
    is partition regular over $s\mathbb{N}$.
\end{lemma}

\begin{proof}
    Given a partition $s\mathbb{N} = \bigcup_{i = 1}^rC_i$, we let $\mathbb{N} = \bigcup_{i = 1}^r(C_i\cap as\mathbb{N})/as$ be a partition of $\mathbb{N}$. By Theorem \ref{PolynomialSzemeredi}, we see that there exists $1 \le i_0 \le r$ and $x,y,z \in (C_{i_0}\cap as\mathbb{N})/as$ for which
	
	\begin{equation}
		x-y = ca^{n-1}s^nz^{n+1}.
	\end{equation}
	The desired result in this case follows from the fact that $asx, asy,asz \in C_{i_0}$ and 
	
	\begin{equation}
		a(asx)-a(asy) = c(asz)^{n+1}.
	\end{equation}
\end{proof}

We are now ready to prove Corollary \ref{TrivialResultForCompleteness}, which implies Theorem 1(i)(b).

\begin{corollary}\label{TrivialResultForCompleteness}
     For any $a,c \in \mathbb{Z}\setminus\{0\}$ and $m,n,s \in \mathbb{N}$ the equation
     
     \begin{equation}
         ax-ay = cw^mz^n
     \end{equation}
     is partition regular over $s\mathbb{N}$.
\end{corollary}

\begin{proof}
A consequence of Lemma \ref{ScaledBergelsonLemma} is that the equation $ax-ay = cz^{m+n}$ is partition regular over $s\mathbb{N}$, so the desired result follows from taking $w = z$.
\end{proof}

We now provide a simple lemma that will not be used later on in the paper, but helps provide some context for items (ii) and (iii) of Theorem \ref{MainResult}.

\begin{lemma}\label{BasicLemma}
    Given $a,c \in \mathbb{Z}\setminus\{0\}$ and $n \in \mathbb{N}$, the equation
    
    \begin{equation}\label{StartingSituation}
        ax = cwz^n
    \end{equation}
    is partition regular over $\mathbb{Q}\setminus\{0\}$ if and only if $\frac{a}{c}$ is an $n$th power in $\mathbb{Q}$.
\end{lemma}

A short proof of Lemma \ref{BasicLemma} can be obtained through the
use of Theorem 3 of \cite{InhomogeneousRado} by viewing the multiplicative group $\mathbb{Q}\setminus\{0\}$ as a $\mathbb{Z}$-module. We choose to give a slightly longer proof since it familiarizes the reader with techniques that will be used repeatedly throughout the rest of the paper. To this end, we begin by recalling the definition of the $p$-adic valuation.

\begin{definition}
Let $p \in \mathbb{N}$ be a prime. The \textit{$p$-adic valuation} is given by $v_p(n) = k$, where $n = p^km$ and $p \nmid m$ for $n \in \mathbb{Z}\setminus\{0\}$. Since $v_p(mn) = v_p(m)+v_p(n)$ for all $m,n \in \mathbb{Z}\setminus\{0\}$, we extend $v_p$ to a well-defined function on $\mathbb{Q}\setminus\{0\}$ by setting $v_p\left(\frac{t}{s}\right) = v_p(t)-v_p(s)$ for all $\frac{t}{s} \in \mathbb{Q}\setminus\{0\}$.
\end{definition}

\begin{proof}[Proof of Lemma \ref{BasicLemma}]
    For the first direction, let us assume that $\frac{a}{c}$ is not an $n$th power in $\mathbb{Q}$, and let $p$ be a prime for which $n\nmid v_p(\frac{a}{c})$. Let $\mathbb{Q}\setminus\{0\} = \bigcup_{i = 1}^nC_i$ be the partition given by
    
    \begin{equation}
        C_i = \left\{\frac{t}{s} \in \mathbb{Q}\setminus\{0\}\ |\ v_p\left(\frac{t}{s}\right) \equiv i\pmod{n}\right\}.
    \end{equation}
    We see that if $w,x,z \in C_{i_0}$ for some $1 \le i_0 \le n$, then 
    
    \begin{equation}
        v_p(ax)-v_p\left(cwz^n\right) \equiv v_p\left(\frac{a}{c}\right)+v_p(x)-v_p(w)-nv_p(z) \equiv v_p\left(\frac{a}{c}\right) \not\equiv 0\pmod{n},
    \end{equation}
    so we cannot have $ax = cwz^n$. 
    
    For the next direction, let us assume that $\frac{a}{c} = (\frac{u}{v})^n$ for some coprime $u,v \in \mathbb{Z}\setminus\{0\}$. Let $p$ be an ultrafilter satisfying the conditions of Theorem \ref{SpecialUltrafilter} and let $\mathbb{Z}\setminus\{0\} = \bigcup_{i = 1}^rC_i$ be a partition. By condition (i) of Theorem \ref{SpecialUltrafilter}, we see that for every $A \in p$, there exists $x,w,z \in A$ for which $x = wz^n$. We observe that
    
    \begin{equation}
        v\mathbb{Z}\setminus\{0\} = \bigcup_{i = 1}^r\frac{v}{u}(C_i\cap u\mathbb{Z}\setminus\{0\})
    \end{equation}
    is a partition, so we may assume without loss of generality that $\frac{v}{u}(C_1\cap u\mathbb{Z}\setminus\{0\}) \in p$. It follows that there exist $x,w,z \in C_1\cap u\mathbb{Z}\setminus\{0\}$ for which
    
    \begin{equation}
        \frac{v}{u}x = \left(\frac{v}{u}w\right)\left(\frac{v}{u}z\right)^n\text{; hence }x = \left(\frac{v}{u}\right)^nwz^n = \frac{c}{a}wz^n\text{, and thus }ax = cwz^n.
    \end{equation}
    The desired result follows from $\mathbb{Z}\setminus\{0\} \subseteq \mathbb{Q}\setminus\{0\}$.
\end{proof}

Our next result, Lemma \ref{PositiveResultForPowers}, is the basis for proving Theorem \ref{MainResult}(ii). While Lemma \ref{PositiveResultForPowers} is an immediate corollary of Theorem 2.11 of \cite{NonStandardRamseyTheory}, we decide to give an independent proof for the sake of completeness and to further familiarize the reader with methods that will be used later on in this paper.

\begin{lemma} \label{PositiveResultForPowers}
	Let $p \in \beta\mathbb{N}$ be an ultrafilter satisfying the conditions of Theorem \ref{SpecialUltrafilter}. For a set $A \in p$, integers $a,b \in \mathbb{Z}\setminus\{0\}$, and $n \in \mathbb{N}$, the equation
	\begin{equation} \label{FirstPositiveCase}
		ax+by = cwz^n
	\end{equation}
	has a solution in $A$ if $c \in \{a,b,a+b\}$.
\end{lemma}

\begin{proof}
	Let 
	
	\begin{equation} A' = \left\{v \in A\ |\ v = wz^n\text{ for some }z,w \in A\right\}.\end{equation}
	Since $A \in p$, to see that $A' = A\setminus(A\setminus A') \in p$ it suffices to observe that $A\setminus A' \notin p$ because $A\setminus A'$ does not satisfy condition $(i)$ of Theorem \ref{SpecialUltrafilter}. Our first case is when $c = a+b$, and in this case we let $x \in A'$ be arbitrary and let $w,z \in A$ be such that $x = wz^n$. Since
	
	\begin{equation}
	    ax+bx = cx = cwz^n,
	\end{equation}
	we see that $(x,x,w,z)$ is a solution to Equation \eqref{FirstPositiveCase} coming from $A$. For our second case it suffices to consider $c = a$ since the case of $c = b$ is handled similarly. By replacing $a,b,c$ with $-a,-b,-c$ if necessary, we may assume without loss of generality that $a > 0$. Observe that $A'_a := A'\cap a\mathbb{N} \in p$ since $A',a\mathbb{N} \in p$ and consider
	
	\begin{alignat}{2}
	    A'' = \Big\{x_1 \in A'_a\ |\ &\text{there exists }x_2 \in A'_a\text{ for which }\numberthis\label{SolutionsViaMPC}\\
	    &x_1+\lambda \frac{x_2}{a} \in A'_a\ \text{for all}\ \lambda \in [-|b|,|b|]\Big\}.
	\end{alignat}
	Since $A'_a \in p$, to see that $A'' = A'_a\setminus(A'_a\setminus A'') \in p$ it suffices to observe that $A'_a\setminus A'' \notin p$ because $A'_a\setminus A''$ does not satisfy condition $(ii)$ of Theorem \ref{SpecialUltrafilter} with $(h,\ell,a,d) = \left(a,b,\frac{x_1}{a},\frac{x_2}{a}\right)$. Now let $x_1 \in A''$ be arbitrary and let $x_2 \in A'_a$ be as in Equation \eqref{SolutionsViaMPC}. Then we observe that
	
	\begin{equation}
	    ax_1+bx_2 = a\left(x_1+b\frac{x_2}{a}\right).
	\end{equation}
	Since $x_1+b\frac{x_2}{a} \in A'$, we may pick $w,z \in A$ for which $x_1+b\frac{x_2}{a} = wz^n$.  In this case we observe that
	
	\begin{equation}
	    ax_1+bx_2 = c\left(x_1+b\frac{x_2}{a}\right) = cwz^n,
	\end{equation}
	so $(x_1,x_2,w,z)$ is a solution to Equation \eqref{FirstPositiveCase} coming from $A$.
\end{proof}

We are now ready to prove Theorem \ref{GeneralPositiveResult}, which implies Theorem \ref{MainResult}(ii).

\begin{theorem}\label{GeneralPositiveResult}
	If $a,b,c \in \mathbb{Z}\setminus\{0\}$ and $n \in \mathbb{N}$ are such that one of $\frac{a}{c}, \frac{b}{c}$, or $\frac{a+b}{c}$ is an $n$th power in $\mathbb{Q}_{\ge 0}$, then the equation
	\begin{equation} \label{GeneralPositiveResultEquation}
		ax+by = cwz^n
	\end{equation}
	\noindent is partition regular over $s\mathbb{N}$ for any $s \in \mathbb{N}$. If one of $\frac{a}{c},\frac{b}{c},$ or $\frac{a+b}{c}$ is an $n$th power in $\mathbb{Q}$, then Equation \eqref{GeneralPositiveResultEquation} is partition regular over $s\mathbb{Z}\setminus\{0\}$ for any $s \in \mathbb{N}$. 
\end{theorem}

\begin{proof}
	Let $d \in \{a,b,a+b\}$ be such that $\frac{d}{c} = (\frac{u}{v})^n$ with $u,v \in \mathbb{Z}$. We see that if $d = 0$, then $a = -b$ and the desired result follows from Corollary \ref{TrivialResultForCompleteness}. Let us now assume that $d \neq 0$, so we also have that $u \neq 0$. Using Lemma \ref{PositiveResultForPowers} we see that if $p \in \beta\mathbb{N}$ is an ultrafilter satisfying the properties of Theorem \ref{SpecialUltrafilter}, then for any $A \in p$ there exists $w,x,y,z \in A$ for which
	\begin{equation}
		ax+by = dwz^n.
	\end{equation}
    We now consider the cases of $\frac{u}{v} \in \mathbb{Q}^+$ and $\frac{u}{v} \in \mathbb{Q}$ separately. If $\frac{u}{v} \in \mathbb{Q}^+$ and $s\mathbb{N} = \bigcup_{i = 1}^rC_i$ is a partition, then 
    
    \begin{equation}
        vs\mathbb{N} = \bigcup_{i = 1}^r\frac{v}{u}(C_i\cap us\mathbb{N})
    \end{equation}
    is also a partition. Similarly, if $\frac{u}{v} \in \mathbb{Q}$ and $s\mathbb{Z}\setminus\{0\} = \bigcup_{i = 1}^rC_i$ is a partition, then
    
    \begin{equation}
        vs\mathbb{N} = \bigcup_{i = 1}^r\left(\frac{v}{u}(C_i\cap us\mathbb{Z})\cap vs\mathbb{N}\right)
    \end{equation}
    is also a partition. In either case, since $vs\mathbb{N} \in p$, there exists $1 \le i_0 \le r$ for which $A := \frac{v}{u}(C_{i_0}\cap us\mathbb{Z})\cap vs\mathbb{N} \in p$, so there exist $w,x,y,z \in \frac{u}{v}A$ for which 
    
    \begin{equation}
        a\left(\frac{v}{u}x\right)+b\left(\frac{v}{u}y\right) = d\left(\frac{v}{u}w\right)\left(\frac{v}{u}z\right)^n \text{, and hence }ax+by = d\left(\frac{v}{u}\right)^nwz^n = cwz^n.
    \end{equation}
\end{proof}

A particularly aesthetic result arises when we set $n = 1$ in Theorem \ref{GeneralPositiveResult}. 

\begin{corollary}
	For any $a,b,c \in \mathbb{Z}\setminus\{0\}$ the equation
	
	\begin{equation} \label{FirstModification}
		ax+by = cwz
	\end{equation}
	is partition regular over $\mathbb{Z}\setminus\{0\}$.
\end{corollary}

\begin{remark}\label{NvsZRemark}
It is interesting to note that the equation $x+y = -wz$ is partition regular over $\mathbb{Z}\setminus\{0\}$ as a consequence of Theorem \ref{GeneralPositiveResult}, but not over $\mathbb{N}$ due to sign obstructions. We are currently unable to determine whether equations such as $2x-8y = wz^3$ are partition regular over $\mathbb{N}$ since there are no sign obstructions preventing the partition regularity. 
\end{remark}

Now that we have proven (ii) of Theorem \ref{MainResult}, we are ready to state Theorem \ref{MainToolForNegativeResults}, which will be a crucial tool in our efforts to prove item (iii) of Theorem \ref{MainResult}. The techniques that we use to prove Theorem \ref{MainToolForNegativeResults} are similar to techniques used in \cite{NonlinearRado}, \cite{NonStandardRamseyTheory}, and \cite{RadosTheorem}, to show that certain equations are not partition regular over $\mathbb{N}$. If $p \in \mathbb{N}$ is a prime and $r,s \in \mathbb{Z}$ are such that $p \nmid s$, then we define $\frac{r}{s} \equiv rs^{-1}\pmod{p}$.

\begin{theorem}\label{NotPartitionRegularCriterion} \label{MainToolForNegativeResults}
	Given $a,b,c\in \mathbb{Z}\setminus\{0\}$ and $n \in\mathbb{N}$, the equation
	\begin{equation} \label{EquationForTheNegativeCase}
		ax+by = cwz^n
	\end{equation}
	is not partition regular over $\mathbb{Q}\setminus\{0\}$ if there exists a prime $p > \max(|a|+|b|,|c|)$ for which $\frac{a}{c}, \frac{b}{c}$, and $\frac{a+b}{c}$ are not $n$th powers mod $p$.
\end{theorem}

\begin{proof}
	Let $p > \text{max}(|a|+|b|,|c|)$ be a prime for which $a, b$, and $a+b$ are not $n$th powers modulo $p$. Let $\chi:\QQ\setminus\{0\}\rightarrow[1,p-1]$ be given by 
	\begin{equation}
		\frac{x}{p^{v_p(x)}} \equiv \chi(x)\pmod{p}.
	\end{equation}
	Note that for all $r,s \in \mathbb{Z}\setminus\{0\}$ we have $\chi(rs) \equiv \chi(r)\chi(s) \pmod{p}$ and for all nonzero $-p < r < p$ we have $r \equiv \chi(r) \pmod{p}$. We also see that for all $r,s \in \mathbb{Z}\setminus\{0\}$ we have 
	
	\begin{equation}
	    \chi(r+s) \equiv \begin{cases}
	                         \chi(r)+\chi(s)\hfill \pmod{p} & \text{ if }v_p(r) = v_p(s)\text{ and }r+s \not\equiv 0\pmod{p}\\
	                         \chi(s)\hfill \pmod{p} & \text{ if }v_p(r) > v_p(s)\\
	                         \chi(r)\hfill \pmod{p} & \text{ if } v_p(s) > v_p(r).
	                     \end{cases}
	\end{equation}
	Let $\QQ\setminus\{0\} = \bigcup_{i = 1}^{p-1}C_i$ be the partition given by $C_i = \chi^{-1}(\{i\})$. Let us assume for the sake of contradiction that there exists $d \in [1,p-1]$ and $w,x,y,z \in C_d$ satisfying Equation \eqref{EquationForTheNegativeCase}. We now have 3 cases to consider. If $v_p(x) = v_p(y)$, then we see that
	\begin{alignat}{2}
		0 \not\equiv &(a+b)d \equiv \chi(a)\chi(x)+\chi(b)\chi(y)\numberthis\\
		\equiv &\chi(ax+by) \equiv \chi(cwz^n) \equiv cd^{n+1} \pmod{p}\text{, and hence}\\
	    &(a+b)c^{-1} \equiv d^n\pmod{p},
	\end{alignat}
	which yields the desired contradiction in this case. For our next case we assume that $v_p(x) < v_p(y)$ and note that
	\begin{alignat}{2}
		0 \not\equiv &ad \equiv \chi(a)\chi(x) \equiv \chi(ax+by) \equiv \chi(cwz^n) \equiv cd^{n+1} \pmod{p}\text{, and hence}\numberthis\\
	    & ac^{-1} \equiv d^n\pmod{p},
	\end{alignat}
	which once again yields a contradiction. Similarly, in our final case when $v_p(x) > v_p(y)$ we have
	\begin{alignat}{2}
		0 \not\equiv &bd \equiv \chi(b)\chi(y) \equiv \chi(ax+by) \equiv \chi(cwz^n) \equiv cd^{n+1}\pmod{p}\text{, and hence}\numberthis\\
		& bc^{-1}\equiv d^n\pmod{p},
	\end{alignat}
	which once more yields a contradiction.
\end{proof}

\begin{manualtheorem}{\ref{MainResult}(iii)(a-b)}
    \label{MainNegativeResult}
	Let $n \in \mathbb{N}$ and $a,b,c\in \mathbb{Z}\setminus\{0\}$ be such that either $n$ is odd and none of $\frac{a}{c}, \frac{b}{c}, \frac{a+b}{c}$ are $n$th powers in $\mathbb{Q}$, or $n \neq 4, 8$ is even and none of $\frac{a}{c}, \frac{b}{c},$ or $\frac{a+b}{c}$ are $\frac{n}{2}$th powers in $\mathbb{Q}$. Then the equation
	\begin{equation}
		ax+by = wz^n
	\end{equation}
	is not partition regular over $\mathbb{Q}\setminus\{0\}$.    
\end{manualtheorem} 

\begin{proof}
	By Theorem \ref{MainToolForNegativeResults} it suffices to construct a prime $p > \text{max}(|a|+|b|,|c|)$ for which none of $\frac{a}{c}, 
	\frac{b}{c},$ or $\frac{a+b}{c}$ are perfect $n$th powers modulo $p$. Firstly, we see that if $n$ is odd, then we may use Theorem 2 to show that the desired prime $p$ exists. Next, we see that if $n = 2m$ with $m$ odd, then none of $\frac{a}{c}, \frac{b}{c},\frac{a+b}{c}$ are $m$th powers by assumption, so we may once again use Theorem 2 to find the a prime $p > \text{max}(|a|+|b|,|c|)$ for which none of $\frac{a}{c}, \frac{b}{c},$ or $\frac{a+b}{c}$ are $m$th powers mod $p$. We note that none of $\frac{a}{c},\frac{b}{c},\frac{a+b}{c}$ are $n$th powers mod $p$ since $m|n$, so $p$ is the desired prime in this case. Lastly, we see that if $4|n$ and $n \ge 12$, then $\frac{n}{4} \ge 3$, so by Fermat's Last Theorem (see \cite{FermatOriginal}, \cite{FermatFixed}), at least one of $\frac{a}{c}, \frac{b}{c},$ or $\frac{a+b}{c}$ is not an $\frac{n}{4}$th power in $\mathbb{Q}$, so we may once again use Theorem 2 to show that the desired prime $p$ exists.
\end{proof}


\section{Number Theoretic Results}
\label{NumberTheorySection}

In this section we assume that the reader has had an introduction to algebraic number theory. Specifically, we assume familiarity with the content appearing in Chapters \rom{1}, \rom{2}, and \rom{4} of \cite{LangNT} and the Chebotarev Density Theorem. The main goal of this section is to prove Theorem 2. The reader willing to take the existence of such primes on faith can safely skip this section and the algebraic number theory content appearing here. We first handle the odd exponent case as some aspects of the argument are simplified and very general. Afterwards, we add a few details to handle the even exponent case. 

We briefly recall some of the concepts we will need. We call $K$ a number field if it is a finite field extension of $\QQ$. We write $\Ocal_K$ for the ring of integers of $K$, which is the integral closure of $\ZZ$ in $K$. This is a Dedekind domain, so nonzero ideals factor uniquely into a product of prime ideals. 

Given an extension of number fields $L/K$, one can ask how a prime ideal $\frakp \subset \Ocal_K$ factors in $\Ocal_L$. We have $\frakp \Ocal_L = \frakq_1^{e_1} \cdots \frakq_g^{e_g}$ for some prime ideals $\frakq_i \subset \Ocal_L$. Recall that $e(\frakq_i/\frakp) := e_i$ is the ramification degree of $\frakq_i$ over $\frakp$ and $f_i = f(\frakq_i/\frakp) := [\Ocal_L/\frakq_i : \Ocal_K/\frakp]$ is the inertia degree of $\frakq_i$ over $\frakp$. We say $\frakp$ is unramified in $\Ocal_L$, or just in $L$, if $e_i = 1$ for all $i$. It is a fact that only finitely many prime ideals of $\Ocal_K$ are ramified in $L$. 

These invariants are bounded via the following classic formula:
\begin{equation}
	[L:K] = \sum_{i=1}^g e_i f_i. 
\end{equation}
In the case $L/K$ is Galois, we have all of the $e_i$'s and $f_i$'s are equal, so in fact $e_i$ and $f_i$ both divide $[L:K]$. 

To compute these numbers in practice, one uses modular arithmetic and factoring polynomials. The process can be summarized as follows. Suppose $L = K(\alpha)$ with $\alpha \in \Ocal_L$. Then $\Ocal_K[\alpha] \subset \Ocal_L$ and both are finite free $\Ocal_K$-modules of rank $[L:K]$, so $[\Ocal_L : \Ocal_K[\alpha]]$ is finite. If the residue characteristic of $\frakp \subset \Ocal_K$ does not divide $[\Ocal_L : \Ocal_K[\alpha]]$, then the factorization behavior of $\frakp$ in $\Ocal_L$ can be detected by factoring the minimal polynomial of $\alpha \bmod \frakp$. 

More precisely, let $f(x) \in \Ocal_K[x]$ be the minimal polynomial for $\alpha$. Then, under the divisibility assumption above, we have 
\begin{equation}
	\frakp \Ocal_L = \frakq_1^{e_1} \cdots \frakq_g^{e_g} \iff f(x) \equiv q_1(x)^{e_1} \cdots q_g(x)^{e_g} \bmod \frakp
\end{equation}
and $\deg(q_i(x)) = f(\frakq_i/\frakp)$. Hence for all but finitely many $\frakp \subset \Ocal_K$, its factoring behavior in $\Ocal_L$ is detected by factoring $f(x) \bmod \frakp$. 

When $L/K$ is Galois with group $G$, there is an important relationship between the arithmetic and algebra of the fields expressed via Frobenius elements. Suppose $\frakq \subset \Ocal_L$ divides $\frakp \subset \Ocal_K$. Then there exists a unique $\Frob_{\frakq/\frakp} \in G$ defined by the property
\begin{equation}
	\Frob_{\frakq/\frakp}(x) \equiv x^q \bmod \frakq
\end{equation}
where $q = \# \Ocal_K/\frakp$. If $\frakq$ and $\frakq'$ are primes dividing $\frakp$ in $\Ocal_L$, then $\Frob_{\frakq/\frakp}$ and $\Frob_{\frakq'/\frakp}$ are conjugate. Conversely, for every $\sigma \in G$ in the conjugacy class of $\Frob_{\frakq/\frakp}$, there exists $\frakq'$ dividing $\frakp$ so that $\sigma = \Frob_{\frakq'/\frakp}$. Thus we can speak of a well-defined Frobenius conjugacy class $\Frob_\frakp \subset G$. When $G$ is abelian, Frobenius elements associated to primes in $K$ are therefore well-defined.

In order for this to be useful to us, we need a way to construct primes with given Frobenius elements. The major tool for achieving this is the Chebotarev density theorem.

\begin{theorem}[Chebotarev Density Theorem]
	Let $L/K$ be a Galois extension of number fields. Let $C \subset G = \Gal(L/K)$ be a fixed conjugacy class. Then the natural density of primes $\frakp \subset K$ with $\Frob_\frakp \in C$ is given by $\# C/ \# G$, i.e.,
	\begin{equation}
		\lim_{x \to \infty} \frac{ \# \{\frakp \in \Ocal_K : \# \Ocal_K/\frakp \le x \mathrm{ , and } \Frob_\frakp \in C\} }{ \# \{ \frakp \in \Ocal_K : \# \Ocal_K/\frakp \le x \}} = \frac{ \# C}{\# G}. 
	\end{equation}
\end{theorem}

Fix $n > 1$. We turn our attention to studying the polynomial $x^n -a$ for $a \in \Ocal_K$ and its factoring behavior modulo various primes in $\Ocal_K$. Suppose $a \in \Ocal_K$ is not an $n$th power. Let $a = \alpha^d$ for $\alpha \in \Ocal_{K(\zeta_n)}$ with $d \mid n$ maximal. Set $m = n/d$. Consider the diagram of fields:
\[
\begin{tikzcd}
	L = C(\alpha^{1/m})               \\
	C = K(\zeta_n) \arrow[u, no head] \\
	K \arrow[u, no head]             
\end{tikzcd}
\]
where $\zeta_n$ is a primitive $n$th root of unity, and $\alpha^{1/m}$ is an arbitrary root of the polynomial $x^m - \alpha$. We will show the field $L$ is a well-defined radical extension of $C$ using the following lemma.

\begin{lemma}[{\cite[Theorem VI.9.1]{LangAlgebra}}]\label{LangCriterion}
	Let $k$ be a field and $a \in k$ nonzero. Assume for all primes $p \mid n$, we have $a \not\in k^p$ and if $4 \mid n$, then $a \not \in -4k^4$. Then $x^n - a$ is irreducible in $k[x]$. 
\end{lemma}

From here we deduce the following.
\begin{lemma}\label{RadicalExtDegree}
	The polynomial $x^m - \alpha$ is irreducible over $C$. In particular, the field $L$ above is a well-defined radical extension of $C$, and $[L:C] = m$. 
\end{lemma}
\begin{proof}
	By the maximality of $d$, we see that for all primes $p \mid n$ we have $\alpha \not\in C^p$, so the polynomial $x^m - \alpha$ satisfies the first criterion of Lemma \ref{LangCriterion}, and the polynomial is irreducible as long as $4 \nmid n$. In the case $4 \mid n$, we check $\alpha \not\in -4C^4$. Suppose otherwise. Then $\alpha = -4\beta^4$ for some $\beta \in C$. But since $4 \mid n$ and $\zeta_n \in C$, we have $\sqrt{-1} = \zeta_4 \in C$, so $\alpha = (\zeta_4 \cdot 2 \cdot \beta^2)^2$, and $a = (\zeta_4 \cdot 2 \cdot \beta^2)^{2d}$, contradicting the maximality of $d$. In either case, we find that $x^m - \alpha$ is irreducible. 
	
	For the statement on degrees, we use Kummer theory. Recall that this tells us that since $C$ contains all $m$th roots of unity, extensions of the form $C(\alpha^{1/m})/C$ are cyclic of degree equal to the order of $\alpha$ in $C^{\times}/C^{\times,m}$. But we have just showed that $\alpha$ is not a $t$th power for any $t \mid m$, so its order in this group is $m$. 
\end{proof}

\begin{remark}\label{NumberTheoryNotation}
The key idea of our argument is as follows. We will use density arguments to produce a prime ideal $\frakp$ in the ring of integers of $C = K(\zeta_n)$ modulo which $x^m - \alpha$ has no root. Given such a $\frakp$, the going down theorem provides a prime ideal of $\Ocal_K$ with the same property. In ffact, if we can bound the density of such $\frakp$ well enough, then repeating with $b$ and $c$ can yield a density bound on the set of prime ideals $\frakp$ where at least one of $a,b,c$ are $n$th powers. If this density is less than 1, then the lemma will be established in that case.

The setup is as follows. Let $a = \alpha^{d_a}; b = \beta^{d_b}$; and $c = \gamma^{d_c}$ with the $d$'s maximal. Set $m_a = n/d_a$, etc. Let $\frakp \subset \Ocal_K$. Then we have that $x^n - a$ has a root modulo $\frakp$ if and only if $x^{m_\alpha} - \alpha$ has a root modulo $\frakp$, and (with finitely many exceptions for $\frakp$) the density of such $\frakp$ correspond to the density of $\frakp$ splitting in $L_a = C(\alpha^{1/m_a})$. By the Chebotarev density theorem, the latter is given by $1/[L_a:C]$. 

Let $\delta_a := 1/[L_a:C] = 1/m_a$ and similarly define $\delta_b$ and $\delta_c$. As mentioned, if $\delta_a + \delta_b + \delta_c < 1$, then there must exist a prime (infinitely many, in fact) in $C$ where none of $a,b,c$ are $n$th powers. Unfortunately this sum can very well be at least 1 or more, so we will devote most of the rest of this section to handling those cases.

First, we gather some results to rule out the case that $\delta_a = 1$, at least when $K = \QQ$ and many other cases. We recall a fact about the interplay between roots of unity and radical extensions, due to Schinzel. Since Lemmas \ref{AbelianRadicalLemma} and \ref{NoDegree1} are generic, we will omit the subscripts and just write $m,d,$ and $L$ in their statements and proofs.
\end{remark}
\begin{lemma}\label{AbelianRadicalLemma}
	Let $\omega_m$ be the number of roots of unity in a field $F$ of characteristic 0. Suppose $x^m - \alpha$ is irreducible over $F$. Then $F(\alpha^{1/m})/F$ is an abelian Galois extension if and only if $\alpha^{\omega_m} = \beta^m$ for some $\beta \in F$. 
\end{lemma}
\begin{proof}
	See \cite[Theorem 2.1]{AbelianBinomials} for the original, or \cite[Theorem 2]{Velez} for another proof.
\end{proof}

This provides us with our first serious condition on $\delta_a$, by restricting the degree $[L:C]$. 

\begin{lemma}\label{NoDegree1}
	Let $\omega_n$ be the number of $n$th roots of unity in $K$. Suppose $m = [L:C] = 1$ with the notation as above. Then 
	\begin{equation}
		a^{\omega_n} = k^n
	\end{equation}
	for some $k \in K$. 
\end{lemma}
\begin{proof}
	Because $C/K$ is a cyclotomic extension, it is abelian. But since $m = 1$, we have $a^{1/n} = \alpha \in C$. Hence $K(a^{1/n})/K$ is abelian. By Lemma \ref{AbelianRadicalLemma}, we then have $a^{\omega_n} = k^n$ for some $k \in K$.
\end{proof}

In particular, we can control the size of $m$ by the assumptions we make on $\alpha$, and hence $a$, in $K$. We will carry the details out towards the end of the proof of the main lemma. For now, we investigate the density of primes where both $a$ and $b$ are $n$th powers.

\begin{lemma}\label{Intersection}
	Let  $m_a$ and $m_b$ be as in Remark \ref{NumberTheoryNotation}. The density of primes $\mathfrak{p}$ in $K(\zeta_n)$ for which both $a$ and $b$ are $n$th powers modulo $\mathfrak{p}$ is at least $1/(m_a \cdot m_b)$. 
\end{lemma}
\begin{proof}
	Since both $C(\alpha^{1/m_a})/C$ and $C(\beta^{1/m_b})/C$ are Galois extensions, their composite field $C(\alpha^{1/m_a}, \beta^{1/m_b})$ is Galois over $C$ and has degree at most $m_a \cdot m_b$. The primes $\mathfrak{p}$ where both $x^{m_a} - \alpha$ and $x^{m_b} - \beta$ have roots modulo $\mathfrak{p}$ correspond to those with trivial Frobenius element in $G = \Gal(C(\alpha^{1/m_a}, \beta^{1/m_b})/C)$. By Chebotarev's density theorem, the set of such primes $\mathfrak{p}$ has density $1/ \# G \ge 1/(m_a \cdot m_b)$. 
\end{proof}

We can now prove the main existence result.

\begin{lemma}
	\label{ExistenceOfPrimes}
	Let $K$ be a number field and $\omega_n$ the number of $n$th roots of unity in $K$. Let $a,b,c \in \Ocal_K$. 
	\begin{enumerate}[(i)]
		\item Suppose $n$ is odd, and that $a^{\omega_{n}}$ is not an $n$th power in $\Ocal_K$, and similarly for $b,c$; or
		\item Suppose $n$ is even, and $a,b,c$ satisfy the same conditions as in (i), but $a^{2 \omega_{n}}$ is also not an $n$th power. 
	\end{enumerate}
	Then there exist infinitely many prime ideals $\mathfrak{p}$ of $K$ for which none of $a,b,c$ are $n$th powers modulo $\mathfrak{p}$.
\end{lemma}
\begin{proof}
	Recall that $C = K(\zeta_n)$ and $L_a = C(\alpha^{1/m_a})$. Let $\delta_{a,b}$ be the density of primes $\mathfrak{p}$ of $C$ for which both $a$ and $b$ are $n$th powers modulo $\mathfrak{p}$. By Lemma \ref{Intersection}, we have $\delta_{a,b} \ge 1/(m_a \cdot m_b)$. Similarly we write $\delta_{a,b,c}$ for the density of primes $\mathfrak{p}$ for which all three are $n$th powers modulo $\mathfrak{p}$. Also let $\Delta$ be the density of primes $\mathfrak{p}$ for which at least one of $a,b,c$ is an $n$th power modulo $\mathfrak{p}$. We want to use inclusion-exclusion to bound $\Delta$. We have
	\[
	\Delta = \delta_a + \delta_b + \delta_c - \delta_{a,b} - \delta_{b,c} - \delta_{a,c} + \delta_{a,b,c}. 
	\]
	Suppose without loss of generality that $\delta_{a,c}$ is minimal among the densities for the possible pairs. Then $\delta_{a,b,c} \le \delta_{a,c}$, so we have
	\begin{align}
		\Delta &\le \delta_a + \delta_b + \delta_c - \delta_{a,b} - \delta_{b,c}, \\
		&\le 1/m_a + 1/m_b + 1/m_c - 1/(m_a \cdot m_b) - 1/(m_b \cdot m_c),
	\end{align}
	using the bound from Lemma \ref{Intersection}. We now split into cases based on the parity of $n$. 
	\begin{enumerate}[(i)]
		\item We handle the case when $n$ is odd first. We see by our assumptions and Lemma \ref{NoDegree1}, we have $m_a, m_b, m_c > 1$. Since $m_a = [L_a:C]$ must divide $n$ and $n$ is odd, we have $m_a, m_b, m_c \ge 3$. So then $\Delta \le 1/3 + 1/3 + 1/3 - 1/9 - 1/9 = 7/9 < 1$. 
		
		\item When $n$ is even, there are a few more exceptional cases for $(\delta_a, \delta_b, \delta_c)$ which we handle now:
		\begin{itemize}
			\item (1/$k$, 1/3, 1/2), $3 \le k \le 6$: From our bound above, we get
			\begin{equation}
				\Delta \le \frac{1}{k}+\frac{1}{3}+\frac{1}{2}-\frac{1}{6}-\frac{1}{2k} = \frac{2}{3}+\frac{1}{2k} < 1.				
			\end{equation} 
			\item (1/$k$, 1/2, 1/2), $k \ge 3$: From our bound above, we get
			\begin{equation}
				\Delta \le \frac{1}{k} + \frac{1}{2} + \frac{1}{2} - \frac{1}{4} - \frac{1}{2k} = 3/4+\frac{1}{2k} < 1. 
			\end{equation}
			\item (1/2, 1/2, 1/2): This case is genuine cause for concern. For example, it is possible to have 3 quadratic extensions of a field where every prime in the base splits in at least one of them (consider $\QQ(\sqrt{a}), \QQ(\sqrt{b})$, and $\QQ(\sqrt{ab}$)). However, we rule this out by the additional assumption on $a$. Suppose $m_a = 2$ so that $\delta_a = 1/2$. Then since $x^n - a = (x^2 - \alpha)f(x)$ for some $f(x) \in C[x]$, we are justified in writing $\alpha = a^{2/n}$. We have $K((a^2)^{1/n})/K$ is abelian since $K((a^2)^{1/n}) = K(\alpha) \subseteq C$, and $C/K$ is abelian as it is a cyclotomic extension. Applying Lemma \ref{AbelianRadicalLemma}, we see that
			\begin{equation}
				a^{2\omega_n} = k^n
			\end{equation}
			for some $k \in K$, contradicting the original assumption on $a$. 
			\item Otherwise, we have already $\delta_a + \delta_b + \delta_c < 1$, so we get $\Delta < 1$. 
		\end{itemize}
	\end{enumerate}
	In all cases we obtain that the density of primes $p$ where at least one of $a,b,c$ is an $n$th power modulo $p$ is less than 1, so there exists a set of primes of positive density where none are $n$th powers.
\end{proof}

In the case $K = \QQ$, this specializes to a more pleasant form. We now prove Theorem 2, an equally pleasant generalization for $\alpha, \beta, \gamma \in \mathbb{Q}$ rather than $a,b,c \in \mathbb{Z}$ as demanded by the applications from Section \ref{PositiveResults}. We remind the reader that if $p \in \mathbb{N}$ is a prime and $r,s \in \mathbb{Z}$ are such that $p\nmid s$, then $\frac{r}{s} \equiv rs^{-1}\pmod{p}$.

\begin{manualtheorem}{\textbf{ 2}}
    Let $\alpha, \beta, \gamma \in \mathbb{Q}$. 
	\begin{enumerate}[(i)]
		\item Suppose $n$ is odd and $\alpha, \beta, \gamma$ are not $n$th powers; or
		\item Suppose $n$ is even, $\alpha,\beta, \gamma$ are not $\frac{n}{2}$th powers, and $\alpha$ is not an $\frac{n}{4}$th power if $4 \mid n$. 
	\end{enumerate}
	Then there exists infinitely many primes $p \in \NN$ modulo which none of $\alpha, \beta, \gamma$ are $n$th powers.
\end{manualtheorem}

\begin{proof}
Let $c \in \mathbb{N}$ be such that for $\alpha ' := c^n\alpha, \beta' := c^n\beta,$ and $\gamma' := c^n\gamma$ we have $\alpha', \beta', \gamma' \in \mathbb{Z}$. Since $\omega_n = 1$ when $n$ is odd and $\omega_n = 2$ when $n$ is even we may apply Lemma \ref{ExistenceOfPrimes} to find infinitely many primes $p \in \mathbb{N}$ for which $\alpha', \beta',$ and $\gamma'$ are not $n$th powers modulo $p$. If $p \in \mathbb{N}$ is a prime for which $p \nmid c$, then $\alpha'$ is an $n$th power modulo $p$ if and only if $\alpha$ is an $n$th power modulo $p$, and similarly for $\beta'$ and $\gamma'$. It follows that there are infinitely many primes $p$ for which $\alpha, \beta,$ and $\gamma$ are also not $n$th powers modulo $p$.
\end{proof}

\begin{remark}\label{FLTRemark}
	When working over $\QQ$, as long as $n \neq 4,8$, when specialized to our situation where $c = a+b$, the condition that one of the three numbers not be an $\frac{n}{4}$th power is automatic by Fermat's Last Theorem. 
	
	Now we give a few more remarks about the subtleties that arose in this proof, and what obstacles prevent pushing it further, both over $\QQ$ and for general number fields $K$. 
\begin{enumerate}[(i)]
	\item It could happen that $x^n - a$ has a root mod all primes $\frakp \subset \Ocal_K$ even though $a$ is not an $n$th power. For example, this happens when $K = \QQ$; the polynomial $x^8 - 16$ has a root mod $p$ for all primes $p \in \ZZ$, but 16 is not an $8$th power (see \cite{Wang}). The Grunwald-Wang theorem says that if $x^n - a$ has a root mod $p$ for all primes $p$, then $n$ must be even of a special form, determined by $K$. This is ruled out above since $16$ is a perfect $8/2 = 4$th power.
	
	\item Even if one were to exclude the situation arising in the Grunwald-Wang theorem, it is possible that $(x^n-a)(x^n-b)$ could have a root mod $p$ for all primes $p$, despite neither factor having this behavior. We consider $3$ such examples when $K = \QQ$, each of which also suggests the necessity of the conditions in item (ii) of Lemma \ref{ExistenceOfPrimes} and item (ii) or Theorem 2.
	
	\begin{enumerate}[(a)]
	\item The polynomial $(x^{12}-3^6)(x^{12}-\beta^4)$ has this feature for any $\beta \in \NN$, because $(x^2+3)(x^3-\beta)$ has a root mod $p$ for all $p \in \ZZ$. In particular, we see that $x^2+3$ has a root modulo $p$ if $p \equiv 0, 1\pmod{3}$ and $x^3-\beta$ has a root modulo $p$ if $p \equiv 2\pmod{3}$. The problem here is that $x^{12} - 3^6$ has a quadratic factor, which lives inside $\QQ(\zeta_{12})$. In the language above, this means that $e = 1$, so the density argument won't work. (see \cite{SmallIntersectivePolynomials})
	
	\item Since $x^8-16 = (x^4-4)(x^4+4)$ and $x^8-16$ has a root modulo $p$ for any prime $p \in \mathbb{N}$, one of $4$ and $-4$ is a $4$th power modulo $p$.

	\item  Since $36$ is a fourth power modulo $p$ for any prime $p \not\equiv 13\pmod{24}$ and $9$ is a fourth power modulo $p$ for any prime $p \equiv 13 \pmod{24}$, we see that $(x^4-36)(x^4-9)$ has a root modulo $p$ for any prime $p \in \mathbb{N}$.
	\end{enumerate}
	For general $K$, one imagines it only gets harder to determine whether $x^n - a$ has a factor whose splitting field lies in a cyclotomic extension. The proof above implies that if $K$ does not have many roots of unity, then there are mild conditions on $a,b,c$ to get the existence of the desired prime.
	
	\item When $K$ doesn't have unique factorization, we lose some of the power afforded by Lemma \ref{NoDegree1} because we cannot ensure that if $a^x = b^y$, then $a$ is a perfect $y/(x,y)$th power. We can pass to ideals generated by $a$ and $b$, and use the fact that $\Ocal_K$ is Dedekind, so its ideals satisfy unique factorization. But this only says that the ideal generated by $a$ is a perfect $y/(x,y)$th power, and this need not imply $a$ has the same property. For example, if $K = \QQ(\sqrt{-5})$ and $\frakp = (2, 1+\sqrt{-5}) \subset \Ocal_K$, then $\frakp^2 = (2)$, but $2 \neq \alpha^2$ for any $\alpha \in \Ocal_K$.
	
	\item When $K$ has lots of units, it can also be difficult to deduce that $a$ is a perfect $y/(x,y)$th power given $a^x = b^y$. For example, the ideal generated by $(-4)$ in $\ZZ$ is the square of the ideal $(2)$, but $-4$ is of course not a perfect square. It is possible one could say more here by trying to control the units of $\Ocal_K$, but this seems a bit intimidating and not of immediate interest. 
\end{enumerate}
\end{remark}
We see that Theorem 2 is not useful when $n = 2$ or $n = 4$, so we will address the case of $n = 2$ separately and observe the implications that it has for the case of $n = 4$. We first require a lemma for constructing primes $p$ for which certain numbers are not squares modulo $p$.

\begin{lemma} \label{Exponent2}
	Suppose that $\alpha, \beta, \gamma \in \mathbb{Q}$ are not squares, and $\alpha\beta\gamma$ is also not a square. There exists a prime $p \in \mathbb{N}$ for which $\alpha, \beta,$ and $\gamma$ are not squares modulo $p$.
\end{lemma}

\begin{proof}
	Suppose on the contrary that no such prime exists. Let $\left( \frac{a}{p} \right)$ be the Legendre symbol and let $S_{\alpha, \beta}$ be the set of primes $p$ for which
	\begin{equation}
		\left( \frac{\alpha}{p} \right) = \left( \frac{\beta}{p} \right) = -1.
	\end{equation}
	By our assumption, we have $\left( \frac{\gamma}{p} \right) = 1$ for all $p \in S_{\alpha, \beta}$ since otherwise we would have produced a prime with the desired features. Using similar notation for the other pairs, we see that $S_{\alpha, \beta} \cap S_{\beta, \gamma} = \varnothing$ because for primes $p$ in the first set, we have $\left( \frac{\alpha}{p} \right) = -1$ but for $p$ in the second set we have $\left( \frac{\alpha}{p} \right) = 1$. 
	
	Thus the density of $S := S_{\alpha, \beta} \cup S_{\beta, \gamma} \cup S_{\gamma, \alpha} \ge 3/4$ as the density of each $S_{\alpha, \beta}$ is at least $1/4$ by Quadratic Reciprocity. But for each $p \in S$, we have
	\begin{equation}
		\left( \frac{\alpha \beta \gamma}{p} \right) = 1.
	\end{equation}
	Hence a set of primes of density $> 1/2$ split in the extension $\QQ(\sqrt{\alpha \beta \gamma})$. But by Quadratic Reciprocity, the set of primes which split in a degree 2 extension has density 1/2. Thus we have $[\QQ(\sqrt{\alpha \beta \gamma}) : \QQ] = 1$, so we must have $\alpha \beta \gamma$ is a square in $\QQ$, contradicting the original assumption. 
\end{proof}

We are now ready to prove Theorem 1(iii)(c), which we state again here in an alternative but equivalent form.

\begin{manualtheorem}{1\textmd{(iii)(c)}}\label{QuadraticCaseNegativeResult}
	If $a,b,c \in \mathbb{Z}\setminus\{0\}$ are such that $\frac{a}{c}, \frac{b}{c}, \frac{a+b}{c},$ and $\left(\frac{a}{c}\right)\left(\frac{b}{c}\right)\left(\frac{a+b}{c}\right)$ are not squares, then for any $n \in \mathbb{N}$ the equation
	\begin{equation}
		ax+by = cwz^{2n}
	\end{equation}
	is not partition regular over $\mathbb{Q}$.
\end{manualtheorem}
\begin{proof}
	We want to use Theorem \ref{NotPartitionRegularCriterion}, so we produce a prime satisfying the conditions there. In particular, we want a prime $p$ such that $\alpha := \frac{a}{c}, \beta := \frac{b}{c}$, and $\gamma := \frac{a+b}{c}$ are not perfect squares in $\ZZ/p\ZZ$. Noting that $\alpha, \beta, \gamma,$ and $\alpha\beta\gamma$ are not squares in $\mathbb{Q}$, we see that the existence of our desired prime $p$ is a consequence of Lemma \ref{Exponent2}.
\end{proof}

We may also use Lemma \ref{Exponent2} to obtain a strengthening of the special case of Theorem 2 in which there are $2$ variables instead of $3$, which will be of use in Section \ref{SectionForSystems} when we determine necessary conditions for some systems of equations to be partition regular.

\begin{lemma}\label{2VariableLemma}
    Let $\alpha, \beta \in \QQ$ and $n \in \NN$. Suppose that one of the following conditions holds:
    \begin{enumerate}[(i)]
        \item $4\nmid n$ and neither of $\alpha, \beta$ are $n$th powers;
        
        \item $4|n$ and neither of $\alpha,\beta$ are $\frac{n}{2}$th powers. 
    \end{enumerate}
     Then there exists infinitely many primes $p \in \NN$ for which neither $\alpha, \beta$ are $n$th powers modulo $p$. 
\end{lemma}
\begin{proof}
We begin by proving the desired result for item (i), and observe that the only case not handled by Theorem 2 is when $n = 2m$ with $m$ odd and at least one of $\alpha, \beta$ is an $m$th power. Suppose $\alpha = x^m$. Write $\beta = y^d$ for $d$ maximal. By similar density arguments as before, the only edge case is when $d = m$, otherwise the sum of densities of primes where at least one is an $n$th power will be strictly less than 1. In the case $d = m$, since $m$ is odd, we see that
\[
\left( \frac{\alpha}{p} \right) = \left( \frac{x}{p} \right)\text{ and }\left( \frac{\beta}{p} \right) = \left( \frac{y}{p} \right).
\]
Letting $\gamma = \alpha = x$ and $\beta = y$, we see that none of $\alpha, \beta, \gamma,$ or $\alpha\beta\gamma = \alpha^2\beta$ are squares in $\mathbb{Q}$, so by Lemma \ref{Exponent2} there exists infinitely many primes $p$ such that $x,y$ are not squares mod $p$, and thus $\alpha, \beta$ are not $n$th powers mod $p$. 

To see that the desired result holds for item (ii) we consider the cases of $n \neq 4$ and $n = 4$ separately. When $n \neq 4$, we let $\gamma \in \mathbb{Q}$ be any element that is not a $\frac{n}{4}$th power and apply Theorem 2. When $n = 4$, we see that neither of $\alpha$ or $\beta$ are squares, so by item (i) there exists infinitely many primes $p \in \mathbb{N}$ modulo which neither of $\alpha,\beta$ are squares.
\end{proof}
We observe that Remark \ref{FLTRemark}(ii)(b) tells us that for any prime $p$ and odd number $m$, one of $4^m$ and $-4^m$ will be a $4m$th power modulo $p$, which justifies our assumptions in item (ii) of Lemma \ref{2VariableLemma}.
\section{Reduction to the Case min(m,n) = 1}
\label{ReductionToM==1}

The main purpose of this section is to show that the equation $ax+by = cw^mz^n$ is not partition regular over $\mathbb{Z}\setminus\{0\}$ if $a,b,c \in \mathbb{Z}\setminus\{0\}$, $a+b \neq 0$, and $n,m \ge 2$. Afterwards, we will show that the equation $ax+by = cwz^n$ is not partition regular over $\mathbb{Z}\setminus\{0\}$ for some values of $a,b,c,$ and $n$ that are not already addressed by Theorem \ref{MainResult}.

\begin{manualtheorem}{1(i)(a)}\label{ReductionTom=1}
	If $a,b,c \in \mathbb{Z}\setminus\{0\}$ and $n,m \in \mathbb{N}$ are such that $a+b \neq 0$ and $n,m \ge 2$, then the equation
	\begin{equation}
		ax+by = cz^nw^m
	\end{equation}
	is not partition regular over $\mathbb{Z}\setminus\{0\}$.
\end{manualtheorem}

In order to prove Theorem \ref{ReductionTom=1} we will use item $(1)$ of Theorem 2.19 of \cite{NonlinearRado}. We now review the relevant definitions from \cite{NonlinearRado} to state and use this result.

\begin{definition}
	Let $\mathbb{N}_0 = \mathbb{N}\cup\{0\}$. Given a polynomial $P\in\mathbb{Z}[x_1,\cdots,x_n]$, we let {\it Supp($P$)} denote the collection of $\alpha \in \mathbb{N}_0^n$ for which
	
	\begin{equation}
		P(x_1,\cdots,x_n) = \sum_{\alpha \in \text{Supp}(P)}c_{\alpha}x^{\alpha},
	\end{equation}
	where $x^{\alpha} = x_1^{\alpha_1}x_2^{\alpha_2}\cdots x_n^{\alpha_n}$, and $c_{\alpha} \neq 0$.
\end{definition}

\begin{example}
	If $P(x,y) = x^2+y^2$, we have $\text{Supp}(P) = \{(2,0),(0,2)\}$.
\end{example}

\begin{example}
	If $P(x,y,z) = xyz+1$, we have $\text{Supp}(P) = \{(0,0,0),(1,1,1)\}$.
\end{example}

\begin{example}
	If $P(w,x,y,z) = wx+y+z^2$, we have $\text{Supp}(P) = \{(1,1,0,0),\allowbreak (0,0,1,0),\allowbreak (0,0,0,2)\}$.
\end{example}

\begin{definition}
	Let $\phi:\mathbb{Z}^n\rightarrow\mathbb{Z}$ be a positive linear map, i.e., \[\phi: (\alpha_1,\cdots,\alpha_n)\mapsto t_1\alpha_1+\cdots+t_n\alpha_n\] with $t_1,\cdots,t_n \in \mathbb{N}_0$.
	\begin{itemize}
		\item If $c$ is a finite coloring of $\mathbb{N}$, then we say that $\phi$ is {\it $c$-monochromatic} if $\{t_1,\cdots\allowbreak,t_n\}$ is $c$-monochromatic;
		
		\item If $P \in \mathbb{Z}[x_1,\cdots,x_n]$, and $(M_0,\cdots,M_{\ell})$ is the increasing enumeration of $\phi(\text{Supp}(P))$, then the {\it partition of $\text{Supp}(P)$ determined by $\phi$} is the ordered tuple $(J_0,\cdots,J_{\ell})$, where $J_i = \{\alpha \in \text{Supp}(P):\phi(\alpha) = M_i\}$. 
	\end{itemize}
\end{definition}

\begin{definition}
	A {\it Rado partition} of $P \in \mathbb{Z}[x_1,\cdots,x_n]$ is an ordered tuple $(J_0,\cdots\allowbreak,J_{\ell})$ such that, for every finite coloring $c$ of $\mathbb{N}$, there exist infinitely many $c$-monochromatic positive linear maps $\phi:\mathbb{Z}^n\rightarrow\mathbb{Z}$ such that $(J_0,\cdots,J_{\ell})$ is the partition of $\text{Supp}(P)$ determined by $\phi$. 
\end{definition}

\begin{definition}
	A {\it Rado set} for $P \in \mathbb{Z}[x_1,\cdots,x_n]$ is a set $J \subseteq \text{Supp}(P)$ such that there exists a Rado partition $(J_0,\cdots, J_{\ell})$ for $P$ such that $J = J_i$ for some $i \in \{0,1,\cdots, \ell\}$.
\end{definition}

\begin{definition}
	An {\it upper Rado functional} of order $\mathcal{M}$ for $P \in \mathbb{Z}[x_1,\cdots,x_n]$ is a tuple $(J_0,\cdots,J_{\ell},d_0,\cdots,d_{\mathcal{M}-1})$ for some $\ell \ge \mathcal{M}$ and $d_0,\cdots,d_{\mathcal{M}-1} \in \mathbb{N}$ such that, for every finite coloring $c$ of $\mathbb{N}$ and for every $r \in \mathbb{N}$, there exist infinitely many $c$-monochromatic positive linear maps $\phi:\mathbb{Z}^n\rightarrow\mathbb{Z}$, $(\alpha_1,\cdots,\alpha_n)\rightarrow t_1\alpha_1+\cdots+t_n\alpha_n$ such that $(J_{\ell},\cdots,J_0)$ is the partition of $\text{Supp}(P)$ determined by $\phi$, and if $(M_{\ell},\cdots,M_0)$ is the increasing enumeration of $\phi(\text{Supp}(P))$, then $M_i-M_\mathcal{M} = d_i$ for $i \in \{0,1,\cdots,\mathcal{M}-1\}$, and $M_\mathcal{M}-M_i \ge r$ for $i \in \{\mathcal{M}+1,\cdots,\ell\}$.
\end{definition}

\begin{definition}
	A polynomial $P(x) = \sum_{\alpha}c_{\alpha}x^{\alpha} \in \mathbb{Z}[x_1,\cdots,x_n]$ satisfies the {\it maximal Rado condition} if for every $q \ge 2$ there exists an upper Rado functional $(J_0,\cdots,J_{\ell},d_0,\cdots,d_{\mathcal{M}-1})$ for $P$ such that, setting $d_\mathcal{M} = 0$, the polynomial
	
	\begin{equation}
		g(w) = \sum_{i = 0}^\mathcal{M}q^{d_i}\sum_{\alpha \in J_i}c_{\alpha}w^{|\alpha|}
	\end{equation}
	has a real root in $[1,q]$.
\end{definition}

\begin{theorem}[\cite{NonlinearRado}, Theorem 2.19(1)]\label{MaximalRadoCondition}
	Fix $P \in \mathbb{Z}[x_1,\cdots,x_n]$. If the equation $P(x_1,\cdots,x_n) = 0$ is partition regular over $\mathbb{N}$, then $P$ satisfies the maximal Rado condition.
\end{theorem}

We are now ready to begin proving Theorem \ref{ReductionTom=1}.

\begin{lemma}\label{NotPROverN}
    If $a,b,c \in \mathbb{Z}\setminus\{0\}$ and $n,m \in \mathbb{N}$ are such that $a+b \neq 0$ and $n,m \ge 2$, then the equation
    
    \begin{equation}
        ax+by = cw^mz^n
    \end{equation}
    is not partition regular over $\mathbb{N}$.
\end{lemma}

\begin{proof}
	Let us fix $a,b,n,m \in \mathbb{N}$ with $n,m \ge 2$ and $a+b \neq 0$. We will begin by showing that the polynomial $P(w,x,y,z) = cz^nw^m-ax-by$ does not have any upper Rado functional of order $\mathcal{M} \ge 1$. From there it will be relatively simple to verify that $P$ does not satisfy the maximal Rado condition. To this end, let us assume for the sake of contradiction that $(J_0,\cdots,J_{\ell},d_0,\cdots,d_{\mathcal{M}-1})$ is an upper Rado functional of order $\mathcal{M}$ for $P$. Note that $\text{Supp}(P) = \{(m,0,0,n),(0,1,0,0),(0,0,1,0)\} = \{M_0',M_1',M_2'\}$. By considering all possible orderings of the set $\{M_0',M_1',M_2'\}$ and the definition of $d_0$, we see that for any finite coloring $c$ of $\mathbb{N}$ there exist infinitely many $c$-monochromatic positive linear maps 
	\begin{equation}
		\phi(\alpha_1,\alpha_2,\alpha_3,\alpha_4) = t_1\alpha_1+t_2\alpha_2+t_3\alpha_3+t_4\alpha_4,
	\end{equation}
	for which at least one of Equations \eqref{FirstTrialEquation}-\eqref{LastTrialEquation} has a solution:
	\begin{alignat}
		\phi\phi(m,0,0,n) & - & \phi(0,1,0,0) = d_0 &\text{ if and only if } \textcolor{white}{-}mt_1 & - & t_2 & + & & & nt_4 & \ =\  & d_0, \label{FirstTrialEquation} \\
		\phi(0,1,0,0) & - & \phi(m,0,0,n) = d_0 &\text{ if and only if } -mt_1 & + & t_2 & - & & & nt_4 & \ =\  & d_0,\\
		\phi(m,0,0,n) & - & \phi(0,0,1,0) = d_0 &\text{ if and only if } \textcolor{white}{-}mt_1 & & & - & t_3 & + & nt_4 & \ =\  & d_0, \\
		\phi(0,0,1,0) & - & \phi(m,0,0,n) = d_0 &\text{ if and only if } -mt_1 & & & + & t_3 & - & nt_4 & \ =\  & d_0, \label{LastHalfDecentTrialEquation} \\
		\phi(0,1,0,0) & - & \phi(0,0,1,0) = d_0 &\text{ if and only if }  & & t_2 & - & t_3 & & & \ =\  & d_0,
	\label{LastTrialEquationNegative} \\
		\phi(0,0,1,0) & - & \phi(0,1,0,0) = d_0 &\text{ if and only if }  & - & t_2 & + & t_3 & & & \ =\  & d_0.
		\label{LastTrialEquation}
	\end{alignat}
	Equivalently, for any finite coloring of $\mathbb{N}$ at least one of Equations \eqref{FirstTrialEquation}-\eqref{LastTrialEquation} must possess infinitely many monochromatic solutions $(t_1,t_2,t_3,t_4)$. By Theorem \ref{InhomogeneousRadoTheorem} we see that any one of Equations \eqref{FirstTrialEquation}-\eqref{LastTrialEquation} is partition regular over $\mathbb{Z}$ (and hence over $\mathbb{N}$) if and only if there exists a constant solution $t_1 = t_2 = t_3 = t_4 = t$. It follows that Equations \eqref{LastTrialEquationNegative} and \eqref{LastTrialEquation} are not partition regular. Furthermore, we see that one of Equations \eqref{FirstTrialEquation}-\eqref{LastHalfDecentTrialEquation} possess infinitely many monochromatic solutions $(t_1,t_2,t_3,t_4)$ in any finite coloring of $\mathbb{N}$ if and only if one of Equations \eqref{FirstTrialEquationReturns}-\eqref{LastHalfDecentTrialEquationReturns} below possess infinitely many monochromatic solutions $(t_1,t_2,t_3,t_4)$ in any finite coloring of $\mathbb{N}$:
	\begin{align}
		m(t_1-\frac{d_0}{m+n-1})-(t_2-\frac{d_0}{m+n-1})+n(t_4-\frac{d_0}{m+n-1}) &= 0,
		\label{FirstTrialEquationReturns} \\
		-m(t_1-\frac{d_0}{m+n-1})+(t_2-\frac{d_0}{m+n-1})-n(t_4-\frac{d_0}{m+n-1}) &= 0, \\
		m(t_1-\frac{d_0}{m+n-1})-(t_3-\frac{d_0}{m+n-1})+n(t_4-\frac{d_0}{m+n-1}) &= 0, \\
		-m(t_1-\frac{d_0}{m+n-1})+(t_3-\frac{d_0}{m+n-1})-n(t_4-\frac{d_0}{m+n-1}) &= 0.
		\label{LastHalfDecentTrialEquationReturns}
	\end{align}
	Since $m,n \ge 2$ we may repeatedly use Corollary \ref{1PREquation} to create a finite partition of $\mathbb{N}$ for which the only monochromatic solution to any of Equations \eqref{FirstTrialEquationReturns}-\eqref{LastHalfDecentTrialEquationReturns} (considered separately, not as a system) is the solution $t_1 = t_2 = t_3 = t_4 = \frac{d_0}{m+n-1}$. It follows that $P$ does not have any upper Rado functionals of order $\mathcal{M} \ge 1$. Noting that an upper Rado functional of order $0$ is just a Rado partition, we note that the set of Rado partitions of $P$ is
	\begin{align}
		\Big\{\Big(\{(m,0,0,n)\},\{(0,1,0,0)\},\{(0,0,1,0)\}\Big), \\ \Big(\{(m,0,0,n)\},\{(0,0,1,0)\},\{(0,1,0,0)\}\Big), \\
		\Big(\{(0,1,0,0)\},\{(m,0,0,n)\},\{(0,0,1,0)\}\Big), \\ \Big(\{(0,1,0,0)\},\{(0,0,1,0)\},\{(m,0,0,n)\}\Big), \\
		\Big(\{(0,0,1,0)\},\{(m,0,0,n)\},\{(0,1,0,0)\}\Big), \\ \Big(\{(0,0,1,0)\},\{(0,1,0,0)\},\{(m,0,0,n)\}\Big), \\
		\Big(\{(0,1,0,0),(0,0,1,0)\},\{(m,0,0,n)\}\Big), \\ \Big(\{(m,0,0,n)\},\{(0,1,0,0),(0,0,1,0)\}\Big)\Big\}.
	\end{align}
	Finally, to see that $P$ does not satisfy the maximal Rado condition, we see that for $q \ge 2$ we have
	\begin{alignat}{2}
		g(w) = &\sum_{i = 0}^mq^{d_i}\sum_{\alpha \in J_i}c_{\alpha}w^{|\alpha|} = \sum_{i = 0}^0q^0\sum_{\alpha \in J_i}c_{\alpha}w^{|\alpha|}\numberthis\\
		= &\sum_{\alpha \in J_0}c_{\alpha}w^{|\alpha|} \in \{cw^{n+m},-aw,-bw,-(a+b)w\},
	\end{alignat}
	so none of the polynomials that $g(w)$ could be, contain a root in $[1,q]$.
\end{proof}

\begin{proof}[Proof of Theorem \ref{ReductionTom=1}]
    By Lemma \ref{NotPROverN} we see that
    
    \begin{equation}
        ax+by = cw^mz^n\label{BadEquation1}
    \end{equation}
    and
    
    \begin{equation}
        ax+by = (-1)^{n+m-1}cw^mz^n\label{BadEquation2}
    \end{equation}
    are not partition regular over $\mathbb{N}$. Let $\mathbb{N} = \bigcup_{i = 1}^{r_1}C_i$ be a partition for which no cell contains a solution to Equation \eqref{BadEquation1} and let $\mathbb{N} = \bigcup_{i = 1}^{r_2}D_i$ be a partition for which no cell contains a solution to Equation \eqref{BadEquation2}. It now suffices to show that no cell of the partition 
    
    \begin{equation}
      \mathbb{Z}\setminus\{0\} = \left(\bigcup_{i = 1}^{r_2}(-D_i)\right)\cup\bigcup_{i = 1}^{r_1}C_i
    \end{equation}
    contains a solution to Equation \eqref{BadEquation1}. It follows from the definition of the $C_i$ that none of them contain a solution to Equation \eqref{BadEquation1}, so let us assume for the sake of contradiction that for some $1 \le i \le r_2$ there exist $w',x',y',z' \in -D_i$ which satisfy Equation \eqref{BadEquation1}. Letting $w = -w', x = -x', y = -y',$ and $z = -z'$, we see that $w,x,y,z \in D_i$ and
    
    \begin{alignat}{2}
         &a(-x)+b(-y) = ax'+by' = c(w')^m(z')^n = c(-w)^m(-z)^n = (-1)^{m+n}cw^mz^n,\\
        &\text{hence } ax+by = (-1)^{m+n-1}cw^mz^n, \numberthis
    \end{alignat}
    which yields the desired contradiction.
\end{proof}

We now describe another condition for a polynomial to be partition regular of a particular flavor, involving lower Rado functionals.

\begin{definition}
	A {\it lower Rado functional} of order $\mathcal{M} \in \mathbb{N}\cup\{0\}$ for $P \in \mathbb{Z}[x_1,\cdots,x_n]$ is a tuple $(J_0,\cdots,J_{\ell},d_1,\cdots,d_\mathcal{M})$ for some $\ell \ge \mathcal{M}$ and $d_1,\cdots,d_\mathcal{M}\in\mathbb{N}$ such that, for every finite coloring $c$ of $\mathbb{N}$ and for every $N \in \mathbb{N}$, there exist infinitely many $c$-monochromatic positive linear maps
	\begin{align}
		\phi:\mathbb{Z}^n &\rightarrow \mathbb{Z}, \\
		(\alpha_1,\cdots,\alpha_n) &\mapsto (t_1\alpha_1+t_2\alpha_2+\cdots+t_n\alpha_n),
	\end{align}
	such that $(J_0,\cdots,J_{\ell})$ is the partition of Supp$(P)$ determined by $\phi$ and, if $(M_0,\cdots\allowbreak,M_\ell)$ is the increasing enumeration of $\phi(\text{Supp}(P))$, then $M_i-M_0 = d_i$ for $i \in \{1,2,\cdots,\mathcal{M}\}$, and $M_{\mathcal{M}+1}-M_\mathcal{M} \ge N$. 
\end{definition}

\begin{definition}
	For $P \in \mathbb{Z}[x_1,x_2,\cdots,x_n]$ and $q \in \mathbb{N}$, the equation \[P(x_1,x_2,\cdots,x_n) = 0\] is {\it $q$-partition regular} if for any $k \in \mathbb{N}$ and any partition $q^k\mathbb{N} = \bigcup_{i = 1}^vA_i$, there exists $1 \le i_0 \le v$ and $y_1,\cdots,y_n \in A_{i_0}$ for which $P(y_1,\cdots,y_n) = 0$.
\end{definition}

\begin{theorem}[{\cite[Theorem 2.19(2)]{NonlinearRado}}]\label{MinimalRadoCondition}
	Suppose that $p \in \mathbb{N}$ is a prime. If $P \in \mathbb{Z}[x_1,\cdots,x_n]$ is $p$-partition regular, then there exists a lower Rado functional $(J_0,\cdots\allowbreak,J_{\ell},d_1,\cdots,d_{\mathcal{M}})$ for $P$ such that setting $d_0 = 0$, the equation
	\begin{equation}
		\sum_{i = 0}^\mathcal{M}p^{d_i}\sum_{\alpha \in J_i}\frac{1}{\alpha!}\frac{\partial^{\alpha}P}{\partial x^{\alpha}}(0,0,\cdots,0)w^{|\alpha|} = 0\label{MinimalRadoConditionEquation}
	\end{equation}
	
	\noindent has an invertible solution in the ring $\mathbb{Z}_p$ of $p$-adic integers.
\end{theorem}

We now provide a lemma on a condition for our polynomial equations to be partition regular.

\begin{lemma} \label{PpartitionRegularLemma}
	Suppose that $a,b,c \in \mathbb{Z}\setminus\{0\}$ and $n \in \mathbb{N}$ are such that the equation
	\begin{equation}
		ax+by = cwz^n
		\label{PpartitionRegularEquation}
	\end{equation}
	is partition regular. If $p$ is a prime for which $v_p(\frac{a+b}{c})\notin n\mathbb{N}\cup\{0\}$, then Equation \eqref{PpartitionRegularEquation} is $p$-partition regular.
\end{lemma}

\begin{proof}
	We will use induction on $k$ to show that Equation \eqref{PpartitionRegularEquation} is partition regular over $p^k\mathbb{N}$ for each $k \ge 0$. The base case of $k = 0$ holds by assumption, so let us proceed to the inductive step and show that the desired result holds for $k+1$ if it holds for $k$. Let $M = v_p(a+b)$ and consider the partition
	\begin{equation}
		p^k\mathbb{N} = \bigcup_{j = 1}^{p^{M+1}}C_j\text{ where }C_j = \{n \in \mathbb{N}\ |\ n \equiv p^kj \pmod{p^{k+M+1}}\}.
	\end{equation}
	Since Equation \eqref{PpartitionRegularEquation} is partition regular over $p^k\mathbb{N}$, let $w,x,y,z \in C_{j_0}$ satisfy Equation \eqref{PpartitionRegularEquation}. We see that
	\begin{alignat}{2}
		& ax+by = cwz^n \text{, so } ax+by \equiv cwz^n\pmod{p^{k+M+1}},\numberthis \\ 
		\text{hence }& aj_0+bj_0 \equiv cj_0^{n+1}p^{nk} \equiv 0\pmod{p^{M+1}}, \\
		\text{thus } & j_0(a+b-cj_0^np^{nk}) \equiv 0\pmod{p^{M+1}}.
	\end{alignat}
	Since $v_p(\frac{a+b}{c}) \notin n\mathbb{N}\cup\{0\}$, we see that 
	
	\begin{equation}
	    v_p\left(a+b-cj_0^np^{nk}\right) = \text{min}\left(v_p(a+b),v_p\left(cj_0^np^{nk}\right)\right) \le M,
	\end{equation} 
	so we must have that $j_0 \equiv 0 \pmod{p}$. We now see that for any partition
	\begin{equation}
		p^{k+1}\mathbb{N} = \bigcup_{j = 1}^rB_j,
	\end{equation}
	we may let $B_{r+1} = p^k\mathbb{N}\setminus p^{k+1}\mathbb{N}$ and construct the partition
	\begin{equation}
		p^k\mathbb{N} = \bigcup_{\underset{\scriptstyle 1 \le j_2 \le p^{M+1}}{\scriptstyle 1 \le j_1 \le r+1\textcolor{white}{'''}}}(B_{j_1}\cap C_{j_2}).
	\end{equation}
	Since Equation \eqref{PpartitionRegularEquation} is partition regular over $p^k\mathbb{N}$, let $w,x,y,z \in B_{J_1}\cap C_{J_2}$ satisfy Equation \eqref{PpartitionRegularEquation}. We have already shown that since $w,x,y,z \in C_{J_2}$, we have $w,x,y,z \in p^{k+1}\mathbb{N}$. Since $w,x,y,z \notin B_{r+1}$, we have shown that Equation \eqref{PpartitionRegularEquation} is also partition regular over $p^{k+1}\mathbb{N}$ as desired.
\end{proof}

This leads us to the following useful criterion. 

\begin{theorem} \label{MinimalRadoConditionApplied}
	Suppose that $a,b,c \in \mathbb{Z}\setminus\{0\}$ and $n \in \mathbb{N}$ are such that the equation
	\begin{equation}
		ax+by = cwz^n
		\label{PpartitionRegularEquation2}
	\end{equation}
	is partition regular. If $p$ is a prime for which $v_p(\frac{a+b}{c})\notin n\mathbb{N}\cup\{0\}$, then one of $\frac{a}{c},\frac{b}{c},\frac{a+b}{c}$ must be an $n$th power in $\mathbb{Q}_p$.
\end{theorem}

\begin{proof}
	We will begin by determining all of the lower Rado functionals for
 \begin{equation}
      P(x_1,x_2,x_3,x_4) = ax_1+bx_2-cx_3x_4^n. 
 \end{equation}
 Since the system of equations
	\begin{equation}
		\begin{array}{ccccccc}
			\phi(1,0,0,0) & = & \alpha_1 & = & \alpha_2 & = & \phi(0,1,0,0) \\
			\phi(0,1,0,0) & = & \alpha_2 & = & \alpha_3+n\alpha_4 & = & \phi(0,0,1,n)
		\end{array}
	\end{equation}
	is partition regular, we see that
	\begin{align}
		\Big\{\Big(\{(1,0,0,0),(0,1,0,0),(0,0,1,n)\}\Big),
		&\Big(\{(1,0,0,0),(0,1,0,0)\},\{(0,0,1,n)\}\Big), \\
		\Big(\{(0,1,0,0),(0,0,1,n)\},\{(1,0,0,0)\}\Big),
		&\Big(\{(1,0,0,0),(0,0,1,n)\},\{(0,1,0,0)\}\Big)\Big\}
	\end{align}
	is the collection of nontrivial\footnote{A lower Rado functional or order $0$ $(J_0,\cdots,J_{\ell})$ is trivial if $J_0$ is a singleton, as such a functional will never yield an invertible solution to Equation \eqref{MinimalRadoConditionEquation}.} lower Rado functionals of order $0$. We now proceed to determine all lower Rado functionals of order $1$. Let $(J_0,\cdots,J_{\ell},d_1)$ be a lower Rado functional of order $1$. Since $d_1 > 0$ we may use Theorem \ref{InhomogeneousRadoTheorem} to see that neither of the equations
	\begin{align}
		d_1 = \phi(1,0,0,0)-\phi(0,1,0,0) &= \alpha_1-\alpha_2\text{, and} \label{bad1} \\
		d_1 = \phi(0,1,0,0)-\phi(1,0,0,0) &= \alpha_2-\alpha_1 \label{bad2}
	\end{align}
	are partition regular. It follows that any lower Rado functional of order $1$ must have $(0,0,1,n) \in J_0\cup J_1$. We also note by Theorem \ref{InhomogeneousRadoTheorem} that the equations
	\begin{align}
		d_1 = \phi(1,0,0,0)-\phi(0,0,1,n) &= \textcolor{white}{-}\alpha_1-\alpha_3-n\alpha_4, \\
		d_1 = \phi(0,1,0,0)-\phi(0,0,1,n) &= \textcolor{white}{-}\alpha_2-\alpha_3-n\alpha_4, \\
		d_1 = \phi(0,0,1,n)-\phi(1,0,0,0) &= -\alpha_1+\alpha_3+n\alpha_4, \\
		d_1 = \phi(0,0,1,n)-\phi(0,1,0,0) &= -\alpha_2+\alpha_3+n\alpha_4,
	\end{align}
	are partition regular over $\mathbb{Z}$ if and only if $n|d_1$. This results in the following list of lower Rado functionals of order $1$: 
	\begin{align}
		\Big\{\Big(\{(1,0,0,0)\},\{(0,0,1,n)\},\{(0,1,0,0)\},nd\Big),
		\\ \Big(\{(0,1,0,0)\},\{(0,0,1,n)\},\{(1,0,0,0)\},nd\Big),\\ 
		\Big(\{(0,0,1,n)\},\{(1,0,0,0)\},\{(0,1,0,0)\},nd\Big),
		\\ \Big(\{(0,0,1,n)\},\{(0,1,0,0\},\{(1,0,0,0)\},nd\Big),\\ 
		\Big(\{(1,0,0,0),(0,1,0,0)\},\{(0,0,1,n)\},nd\Big),
		\\ \Big(\{(0,0,1,n)\},\{(1,0,0,0),(0,1,0,0)\},nd\Big)\Big\}.
	\end{align}
	Lastly, we recall that the only lower Rado functionals of order $2$ are of the form $(J_0,J_1,J_2,d_1,d_2)$, but such a lower Rado functional cannot exist since Equations \eqref{bad1} and \eqref{bad2} are not partition regular. We have now determined all of the lower Rado functionals for $P(x_1,x_2,x_3,x_4)$. By Lemma \ref{PpartitionRegularLemma} we see that Equation \eqref{PpartitionRegularEquation2} is $p$-partition regular, so we may apply Theorem \ref{MinimalRadoCondition} to see that at least $1$ of Equations \eqref{MinimalRadoStart}-\eqref{MinimalRadoEnd} has an invertible solution in $\mathbb{Z}_p$. After each equation we give the corresponding lower Rado functional in parenthesis.
	\begin{align}
		(a+b)w-cw^{n+1} &= 0\ \Big(\{(1,0,0,0),(0,1,0,0),(0,0,1,n)\}\Big), \label{MinimalRadoStart} \\
		(a+b)w-cp^{nd}w^{n+1} &= 0\ \Big(\{(1,0,0,0),(0,1,0,0)\},\{(0,0,1,n)\},nd\Big), \\
		(a+b)p^{nd}w-cw^{n+1} &= 0\ \Big(\{(0,0,1,n)\},\{(1,0,0,0),(0,1,0,0)\},nd\Big), \\
		aw-cw^{n+1} &= 0\ \Big(\{(1,0,0,0),(0,0,1,n)\},\{(0,1,0,0)\}\Big), \\
		aw-cp^{nd}w^{n+1} &= 0\ \Big(\{(1,0,0,0)\},\{(0,0,1,n)\},\{(0,1,0,0)\},nd\Big), \\
		ap^{nd}w-cw^{n+1} &= 0\ \Big(\{(0,0,1,n)\},\{(1,0,0,0)\},\{(0,1,0,0)\},nd\Big), \\
		bw-cw^{n+1} &= 0\ \Big(\{(0,1,0,0),(0,0,1,n)\},\{(1,0,0,0)\}\Big), \\
		bw-cp^{nd}w^{n+1} &= 0\ \Big(\{(0,1,0,0)\},\{(0,0,1,n)\},\{(1,0,0,0)\},nd\Big), \\
		bp^{nd}w-cw^{n+1} &= 0\ \Big(\{(0,0,1,n)\},\{(0,1,0,0\},\{(1,0,0,0)\},nd\Big), \\
		(a+b)w &= 0\ \Big(\{(1,0,0,0),(0,1,0,0)\},\{(0,0,1,n)\}\Big). \label{MinimalRadoEnd} 
	\end{align}
	The desired result follows after noting that one of Equations \eqref{MinimalRadoStart}-\eqref{MinimalRadoEnd} has an invertible solution in $\mathbb{Z}_p$ if and only if one of $\frac{a}{c},\frac{b}{c},\frac{a+b}{c}$ is an $n$th power in $\mathbb{Q}_p$.
\end{proof}
Before using Theorem \ref{MinimalRadoConditionApplied}, let us recall when $a \in \mathbb{Z}_2$ is an $n$th power. If $a = 2^km$ with $m$ odd, it is a well-known consequence of Hensel's lemma that $a$ is a $2^n$th power in $\ZZ_2$ if and only if $2^n \mid k$ and $m \equiv 1 \pmod{2^{n+2}}$.

\begin{corollary} \label{SomeMoreExamples}
	The following equations are not partition regular as seen by an application of Theorem \ref{MinimalRadoConditionApplied} with $p = 2$ for items (i)-(iii), $p = 3$ for item (iv), and $p = 5$ for item (v):
	
	\begin{itemize}
		\item[(i)] $3x+13y = wz^8$;
		
		\item[(ii)] $4x-8y = wz^4$;
		
		\item[(iii)] $3\cdot5\cdot2^2x+2\cdot5\cdot3^2y = wz^2$;
		
		\item[(iv)] $3^4x+3^6y = wz^{12}$;
		
		\item[(v)] $(3^2\cdot4\cdot5)^2x+(3\cdot4^2\cdot5)^2y = wz^4$.
	\end{itemize}
\end{corollary}

\begin{remark}
We observe that the examples in Corollary \ref{SomeMoreExamples} are of interest because the fact that they are not partition regular cannot be deduced from Theorem \ref{MainToolForNegativeResults}. We will now verify this claim for each of (i)-(v). For (i), we recall that $16$ is an $8$th power modulo $p$ for every prime $p$. For (ii), we recall that for any prime $p$, at least one of $4$ or $4-8=-4$ will be a $4$th power modulo $p$. For (iii), we observe that one of $3\cdot5\cdot2^2, 2\cdot5\cdot3^2$, or $3\cdot5\cdot2^2+2\cdot5\cdot3^2 = 2\cdot3\cdot5^2$ is a square modulo $p$ for every prime $p$. For (iv) we observe that $-3$ is a square modulo $p$ if $p \equiv 1\pmod{3}$ and $3$ is a cube modulo $p$ if $p \equiv 2\pmod{3}$, so either $3^6$ or $3^4$ is a $12$th power modulo any prime $p$. For (v), we observe that $\alpha = (3^2\cdot4\cdot5)^2, \beta = (3\cdot4^2\cdot5)^2$, and $\gamma = \alpha+\beta = (3\cdot4\cdot5^2)^2$ are all squares but are not fourth powers. Since one of $3^2\cdot4\cdot5, 3\cdot4^2\cdot5$, and $3\cdot4\cdot5^2$ will be a perfect square modulo any prime $p$, we see that one of $\alpha, \beta,$ or $\gamma$ will be a perfect fourth power modulo $p$ for any prime $p$.
\end{remark}

\section{On the Partition Regularity of \texorpdfstring{$ax+by = cwz^n$}{ax+by=cwzn} over Integral Domains}\label{SectionForGeneralDomains}

The purpose of this section is to try and generalize as much of Theorem \ref{MainResult} as we can to the more general setting of integral domains instead of just $\mathbb{Z}$ or $\mathbb{N}$.

\begin{theorem}[cf. Theorem \ref{SpecialUltrafilter5}]\label{SpecialUltrafilter6}
Let $R$ be an integral domain. There exists an ultrafilter $p \in \beta R\setminus\{0\}$ with the following properties.
\begin{enumerate}[(i)]
    \item For any $A \in p$ and $\ell \in \mathbb{N}$, there exists $b,g \in A$ with $\left\{bg^j\right\}_{j = 0}^{\ell} \subseteq A$.
    
    \item For any $A \in p$, $\ell \in \mathbb{N}$, and $h,s \in R\setminus\{0\}$, there exists $a,d \in R$ for which $\{hd,ha,ha+sd\} \subseteq A$.
    
    \item For every $\alpha \in R\setminus\{0\}$, we have $\alpha R \in p$.
\end{enumerate}
\end{theorem}

\begin{lemma} \label{PositiveResultForPowersInDomains}
	Let $R$ be an integral domain and let $p \in \beta R\setminus\{0\}$ be an ultrafilter satisfying the conditions of Theorem \ref{SpecialUltrafilter6}. For any $A \in p$, $a,b \in R\setminus\{0\}$ and $n \in \mathbb{N}$, the equation
	\begin{equation} \label{FirstPositiveCaseForDomains}
		ax+by = cwz^n
	\end{equation}
	has a solution in $A$ if $c \in \{a,b,a+b\}$.
\end{lemma}

\begin{proof}
	Let 
	
	\begin{equation} A' = \{v \in A\ |\ v = wz^n\text{ for some }z,w \in A\}.\end{equation}
	Since $A \in p$, to see that $A' = A\setminus(A\setminus A') \in p$ it suffices to observe that $A\setminus A' \notin p$ because $A\setminus A'$ does not satisfy condition $(i)$ of Theorem \ref{SpecialUltrafilter6}. Our first case is when $c = a+b$, and in this case we let $x \in A'$ be arbitrary and let $w,z \in A$ be such that $x = wz^n$. Since
	
	\begin{equation}
	    ax+bx = cx = cwz^n,
	\end{equation}
	we see that $x,x,w,z$ is a solution to Equation \eqref{FirstPositiveCaseForDomains} coming from $A$. For our second case it suffices to consider $c = a$ since the case of $c = b$ is handled similarly. Observe that $A'_a := A'\cap aR \in p$ since $A',aR \in p$ and consider
	
	\begin{equation}\label{SolutionsViaMPCForDomains}
	    A'' = \left\{x_1 \in A'_a\ |\ \text{there exists }x_2 \in A'_a\text{ satisfying }x_1+b \frac{x_2}{a} \in A'_a\right\}.
	\end{equation}
	Since $A'_a \in p$, to see that $A'' = A'_a\setminus(A'_a\setminus A'') \in p$ it suffices to observe that $A'_a\setminus A'' \notin p$ because $A'_a\setminus A''$ does not satisfy condition $(ii)$ of Theorem \ref{SpecialUltrafilter6} with $(h,s,a,d) = (a,b,\frac{x_1}{a},\frac{x_2}{a})$. Now let $x_1 \in A''$ be arbitrary, let $x_2 \in A'_a$ be as in Equation \eqref{SolutionsViaMPCForDomains}, and observe that
	
	\begin{equation}
	    ax_1+bx_2 = a\left(x_1+b\frac{x_2}{a}\right).
	\end{equation}
	Since $x_1+b\frac{x_2}{a} \in A'$, we may pick $w,z \in A$ for which $x_1+b\frac{x_2}{a} = wz^n$.  In this case we observe that
	
	\begin{equation}
	    ax_1+bx_2 = c\left(x_1+b\frac{x_2}{a}\right) = cwz^n,
	\end{equation}
	so $x_1,x_2,w,z$ is a solution to Equation \eqref{FirstPositiveCaseForDomains} coming from $A$.
\end{proof}

Before proceeding further let us recall some notation. If $R$ is an integral domain, then for $u,v \in R\setminus\{0\}$ and $A \subseteq R$ we have

\begin{equation}
    \frac{v}{u}A = \left\{r \in R\ |\ \frac{u}{v}r \in A\right\} = \left\{\frac{v}{u}a\ |\ a \in A\cap uR\right\}.
\end{equation}
Similarly, if $p \in \beta R$ is an ultrafilter, then we have

\begin{equation}
    \frac{u}{v}\cdot p = \left\{A \subseteq R\ |\ \frac{v}{u}A \in p\right\} = \left\{\frac{u}{v}A\ |\ A \in p\right\}.
\end{equation}
\begin{theorem}\label{GeneralPositiveResultInDomains}
Let $R$ be an integral domain with field of fractions $K$ and let $p \in \beta R\setminus\{0\}$ be an ultrafilter satisfying the conditions of Theorem \ref{SpecialUltrafilter6}. If $a,b,c \in R\setminus\{0\}$ and $n \in \mathbb{N}$ are such that one of $\frac{a}{c},\frac{b}{c},$ or $\frac{a+b}{c}$ is of the form $(\frac{u}{v})^n$ for some $u,v \in R$, then the equation

\begin{equation}\label{PositiveResultForDomainsEquation}
    ax+by = cwz^n
\end{equation}
contains a solution for any $A \in q$ where
\begin{equation}
    q = \begin{cases}
            p & \text{if }u = 0\\
            \frac{u}{v}\cdot p & \text{else}.
        \end{cases}
\end{equation}
In particular, Equation \eqref{PositiveResultForDomainsEquation} is partition regular over $R\setminus\{0\}$.
\end{theorem}

\begin{proof}
We see that if $u = 0$ then $a+b = 0$, so the desired result in this case follows from Lemma \ref{PositiveResultForPowersInDomains}. Now let us assume that $u \neq 0$ and let $d \in \{a,b,a+b\}$ be such that $\frac{d}{c} = (\frac{u}{v})^n$. Let $A \in \frac{u}{v}\cdot p$ be arbitrary and note that $\frac{v}{u}A \in p$. By Lemma \ref{PositiveResultForPowersInDomains} there exist $w,x,y,z \in A$ for which $\frac{v}{u}w,\frac{v}{u}x,\frac{v}{u}y,\frac{v}{u}z \in \frac{v}{u}A$ and

\begin{equation}
    a\left(\frac{v}{u}x\right)+b\left(\frac{v}{u}y\right) = d\left(\frac{v}{u}w\right)\left(\frac{v}{u}z\right)^n \Rightarrow ax+by = d\left(\frac{v}{u}\right)^nwz^n = cwz^n.
\end{equation}
For the latter half of the theorem, it suffices to note that if $R\setminus\{0\} = \bigcup_{i = 1}^rC_i$ is a partition, then there exists $1 \le i_0 \le r$ for which $C_{i_0} \in q$, and hence $C_{i_0}$ contains a solution to Equation \eqref{PositiveResultForDomainsEquation}.
\end{proof}

We recall that if $R$ is a Dedekind domain and $\frakp \subseteq R$ is a prime (hence maximal) ideal, then for any $u \in R$ and $v \in R\setminus\mathfrak{p}$ we have $\frac{u}{v} \equiv uv^{-1}\pmod{\frakp}$.

\begin{theorem} \label{NegativeCaseInNumberFields}
	Let $R$ be a Dedekind domain with field of fractions $K$. Let $a,b,c \in R\setminus\{0\}$ and $n \in \mathbb{N}$ be such that none of $\frac{a}{c}, \frac{b}{c},$ or $\frac{a+b}{c}$ are $n$th powers in $R/\mathfrak{p}$ for some prime ideal $\mathfrak{p} \subseteq R$ satisfying $a,b,a+b,c \notin \mathfrak{p}$ and $[R:\frakp] < \infty$. Let $K_\frakp$ denote the completion of $K$ at $\frakp$. The equation
	\begin{equation} \label{NegativeCaseInNumberFieldsEquation}
		ax+by = cwz^n
	\end{equation}
	is not partition regular over $K_\frakp\setminus\{0\}$.
\end{theorem}

\begin{proof}
	Since $R$ is a Dedekind domain we see that $R_\frakp$ is a discrete valuation ring under the valuation $v_\frakp$, so let $\pi$ be a generator of the maximal ideal of $R_\frakp$. Let $F \subseteq R$ be a set of coset representatives of $(\pi)$ such that $\bigcup_{f \in F}\left(f+(\pi)\right) = R\setminus\frakp$ and $\left(f_1+(\pi)\right)\cap\left(f_2+(\pi)\right) = \emptyset$ whenever $f_1 \neq f_2$. We note that $|F| = [R:\frakp]-1 < \infty$. Let $\chi:K_{\mathfrak{p}}\setminus\{0\}\rightarrow F$ be given by
	
	\begin{equation}
		\frac{x}{\pi^{v_{\mathfrak{p}}(x)}} \equiv \chi(x)\pmod{\mathfrak{p}}.
	\end{equation}
	Note that $\chi(rs) = \chi(r)\chi(s) \pmod{\pi}$ for all $r,s \in R\setminus\{0\}$. We also see that 
	
		\begin{equation}
	    \chi(r+s) \equiv \begin{cases}
	                         \chi(r)+\chi(s)\hfill \pmod{\pi} & \text{ if }v_\frakp(r) = v_\frakp(s)\text{ and }r+s \not\equiv 0\pmod{\pi}\\
	                         \chi(s)\hfill \pmod{\pi} & \text{ if }v_\frakp(r) > v_\frakp(s)\\
	                         \chi(r)\hfill \pmod{\pi} & \text{ if } v_\frakp(s) > v_\frakp(r).
	                     \end{cases}
	\end{equation}
	Let $K_{\mathfrak{p}}\setminus\{0\} = \bigcup_{f \in F}C_f$ be the partition given by $C_f = \chi^{-1}(\{f\})$. Let us assume for the sake of contradiction that there exists $d \in F$ for which $w,x,y,z \in C_d$ and Equation \eqref{NegativeCaseInNumberFieldsEquation} is satisfied. We now have $3$ cases to consider. If $v_{\mathfrak{p}}(x) = v_{\mathfrak{p}}(y)$, then we see that
	\begin{alignat}{2}
		&0 \not\equiv (a+b)d \equiv \chi(a)\chi(x)+\chi(b)\chi(y) \equiv \chi(ax+by) \equiv \chi(cwz^n) \equiv cd^{n+1} \pmod{\mathfrak{p}},\\
	    &\text{hence } (a+b)c^{-1} \equiv d^n\pmod{\mathfrak{p}},\numberthis
	\end{alignat}
    which yields the desired contradiction. For our next case we assume that $v_{\mathfrak{p}}(x) < v_{\mathfrak{p}}(y)$ and note that
	\begin{alignat}{2}
		&0 \not\equiv ad \equiv \chi(a)\chi(x) \equiv \chi(ax+by) \equiv \chi(cwz^n) \equiv cd^{n+1} \pmod{\mathfrak{p}},\numberthis\\
	    &\text{hence }ac^{-1} \equiv d^n\pmod{\mathfrak{p}},
	\end{alignat}
	which again yields the desired contradiction. Similarly, in our final case when $v_{\mathfrak{p}}(x) > v_{\mathfrak{p}}(y)$ we have
	\begin{alignat}{2}
		&0 \not\equiv bd \equiv \chi(b)\chi(y) \equiv \chi(ax+by) \equiv \chi(cwz^n) \equiv cd^{n+1}\pmod{\mathfrak{p}},\numberthis\\
	    &\text{hence }bc^{-1}\equiv d^n\pmod{\mathfrak{p}},
	\end{alignat}
	which once more yields the desired contradiction.
\end{proof}

\begin{corollary} \label{MainNegativeResultForNumberFields}
	Let $K$ be a number field and let $\omega_m$ be the number of $m$th roots of unity in $K$. Let $a,b,c \in \Ocal_K$ and let $n \in \mathbb{N}$. Let $d_a$ be the largest integer for which $a^{\frac{1}{d_a}} \in \mathcal{O}_K$, and define $d_b$ and $d_c$ similarly. Let $m_a = \frac{n}{d_a}, m_b = \frac{n}{d_b},$ and $m_c = \frac{n}{d_c}$. 
	\begin{enumerate}[(i)]
		\item Suppose $n$ is odd, and none of $a^{\omega_{m_a}}, b^{\omega_{m_b}},$ and $c^{\omega_{m_c}}$ are an $n$th power in $\Ocal_K$; or
		\item Suppose $n$ is even, and $a,b,c$ satisfy the same conditions as in (i), but $a^{2 \omega_{m_a}}$ is also not an $n$th power. 
	\end{enumerate}
	Then the equation
	
	\begin{equation}
		ax+by = cwz^n
	\end{equation}
	is not partition regular over $K\setminus\{0\}$.
\end{corollary}

\begin{proof}
	The given assumptions are precisely what we need to use Lemma \ref{ExistenceOfPrimes} and obtain a prime ideal $\mathfrak{p} \subseteq \Ocal_K$ for which none of $a,b,$ and $c$ are $n$th powers modulo $\mathfrak{p}$. After noting that $K$ embeds in $K_{\mathfrak{p}}$, we see that the desired result follows from Theorem \ref{NegativeCaseInNumberFields}.
\end{proof}

\begin{remark}
Consider the equation

\begin{equation}\label{AnExample}
    2x+3y = wz^2.
\end{equation}
Since $2,3,$ and $5$ are not squares modulo $43$, Theorem \ref{MainNegativeResultForNumberFields} tells us that Equation \eqref{AnExample} is not partition regular over $\mathbb{Q}_{43}\setminus\{0\}$. Since Equation \eqref{AnExample} is partition regular over the countable set $\mathbb{Z}[\sqrt{2}]$ as a consequence of Theorem \ref{GeneralPositiveResultInDomains} but not over the uncountable set $\mathbb{Q}_{43}$, we see that the algebraic properties of the underlying set $S$ have a stronger influence on the partition regularity of equations of the form $ax+by = cwz^n$ than the cardinality of $S$.
\end{remark}
\section{Systems of equations}\label{SectionForSystems}

\begin{theorem}\label{PositiveResultForSystems}
Let $R$ be an integral domain with field of fractions $K$ and let\\ $p \in \beta R\setminus\{0\}$ be an ultrafilter satisfying the conditions of Theorem \ref{SpecialUltrafilter6}. If $$a_1,\cdots,a_k,b_1,\cdots,b_k,\allowbreak c_1,\cdots,c_k \in R\setminus\{0\}$$ and $n_1,\cdots,n_k \in \mathbb{N}$ are such that

\begin{equation}
    I := K\cap\bigcap_{i = 1}^k\left\{\sqrt[n_i]{\frac{a_i}{c_i}},\sqrt[n_i]{\frac{b_i}{c_i}},\sqrt[n_i]{\frac{a_i+b_i}{c_i}}\right\} \neq \emptyset,
\end{equation}
then the system of equations

\begin{equation}\label{PRPolynomialSystem}
    \begin{array}{ccccc}
         a_1x_1 & + & b_1y_1 & = & c_1w_1z_1^{n_1}\\
         a_2x_2 & + & b_2y_2 & = & c_2w_2z_2^{n_2}\\
         & & \vdots & & \\
         a_kx_k & + & b_ky_k & = & c_kw_kz_k^{n_k}
    \end{array}
\end{equation}
contains a solution in every $A \in q$, where we may take

\begin{equation}
    q = \begin{cases}
            p &\text{if }0 \in I\\
            i\cdot p &\text{if }i \in I\setminus\{0\}.
        \end{cases}
\end{equation}
In particular, the system of equations in \eqref{PRPolynomialSystem} is partition regular over $R\setminus\{0\}$.
\end{theorem}

\begin{proof}
Since none of the equations in the system of equations in \eqref{PRPolynomialSystem} share any variables, the desired result follows from Theorem \ref{GeneralPositiveResultInDomains}
\end{proof}

\begin{theorem}\label{NegativeResultForSystemsOverDedekindDomains}
Let $R$ be a Dedekind domain with field of fractions $K$. Let $a_1,\cdots,a_k,\allowbreak b_1,\cdots,b_k,c_1,\cdots,c_k \in R\setminus\{0\}$, let $n \in \mathbb{N}$, and let 

\begin{equation}
    I := \bigcap_{i = 1}^k\left\{\frac{a_i}{c_i},\frac{b_i}{c_i},\frac{a_i+b_i}{c_i}\right\}
\end{equation}
Suppose that there exists a prime ideal $\frakp \subseteq R$ satisfying:
\begin{enumerate}[(i)]
\item None of $a_1,\cdots,a_i,b_1,\cdots,b_i,c_1,\cdots,c_i,a_1+b_1,\cdots,a_i+b_i$ are contained in $\frakp$.

\item If $v_1,v_2 \in \bigcup_{i = 1}^k\left\{\frac{a_i}{c_i},\frac{b_i}{c_i},\frac{a_i+b_i}{c_i}\right\}$ are distinct, then $v_1 \not\equiv v_2\pmod{\frakp}$.

\item No element of $I$ is an $n$th power modulo $\frakp$.

\item $[R:\frakp] < \infty$.
\end{enumerate}
The system of equations

\begin{equation}\label{NegativeCaseForSystemsOverDomains}
    \begin{array}{ccccc}
         a_1x_1 & + & b_1y_1 & = & c_1w_1z_1^n\\
         a_2x_2 & + & b_2y_2 & = & c_2w_2z_2^n\\
          & & \vdots & & \\
          a_kx_k & + & b_ky_k & = & c_kw_kz_k^n
    \end{array}
\end{equation}
is not partition regular over $K_\frakp\setminus\{0\}$, where $K_\frakp$ is the localization of $K$ at $\frakp$.
\end{theorem}

\begin{proof}
We begin the proof similarly to that of Theorem \ref{NegativeCaseInNumberFields}. Since $R$ is a Dedekind domain we see that $R_\frakp$ is a discrete valuation ring under the valuation $v_\frakp$, so let $\pi$ be a generator of the maximal ideal of $R_\frakp$. Let $F \subseteq R$ be a set of coset representatives of $(\pi)$ such that $\bigcup_{f \in F}\left(f+(\pi)\right) = R\setminus\frakp$ and $\left(f_1+(\pi)\right)\cap\left(f_2+(\pi)\right) = \emptyset$ whenever $f_1 \neq f_2$. We note that $|F| = [R:\frakp]-1 < \infty$. Let $\chi:K_{\mathfrak{p}}\setminus\{0\}\rightarrow F$ be given by
	
	\begin{equation}
		\frac{x}{\pi^{v_{\mathfrak{p}}(x)}} \equiv \chi(x)\pmod{\mathfrak{p}}.
	\end{equation}
	Observe that $\chi(r) \equiv r\pmod{\frakp}$ for all $r \notin \frakp$. Let $K_{\mathfrak{p}}\setminus\{0\} = \bigcup_{f \in F}C_f$ be the partition given by $C_f = \chi^{-1}(\{f\})$. Let us assume for the sake of contradiction that there exists $d \in F$ and $x_i,y_i,w_i,z_i \in C_d$ for $1 \le i \le k$ for which the system of equations in \eqref{NegativeCaseForSystemsOverDomains} is satisfied.
	We see that for $1 \le i \le k$ we have
	
	\begin{equation}
	    0 \not\equiv \chi(c_iw_iz_i^n) \equiv \chi(a_ix_i+b_iy_i) \equiv \begin{cases}
	                         \chi(a_i)d+\chi(b_i)d \hfill \pmod{\pi} & \text{ if }v_\frakp(x_i) = v_\frakp(y_i)\\
	                         \chi(a_i)d\hfill \pmod{\pi} & \text{ if }v_\frakp(y_i) > v_\frakp(x_i)\\
	                         \chi(b_i)d\hfill \pmod{\pi} & \text{ if } v_\frakp(x_i) > v_\frakp(x_i),
	                     \end{cases}
	\end{equation}
	
	\begin{equation}
	    \text{hence } d^n \equiv d^{-1}\chi(w_iz_i^n) = \begin{cases}
	                         (a_i+b_i)c^{-1} \hfill \pmod{\pi} & \text{ if }v_\frakp(x_i) = v_\frakp(y_i)\\
	                         a_ic^{-1}\hfill \pmod{\pi} & \text{ if }v_\frakp(y_i) > v_\frakp(x_i)\\
	                         b_ic^{-1}\hfill \pmod{\pi} & \text{ if } v_\frakp(x_i) > v_\frakp(x_i).
	                     \end{cases}
	\end{equation}
	For $1 \le i \le k$ let $v_i \in \{a_ic_i^{-1},b_ic_i^{-1},(a_i+b_i)c_i^{-1}\}$ be such that $d^n \equiv v_i\pmod{\frakp}$. Since we must have that $v_i \equiv d^n \equiv v_j \pmod{\frakp}$ for all $1 \le i < j \le k$, we see that there is some $v \in K$ for which $v = v_i$ for all $1 \le i \le k$, and hence $v \in I$. The desired contradictions follows after recalling that no element of $I$ is an $n$th power modulo $\frakp$.
\end{proof}

We observe that the conditions of Theorem \ref{NegativeResultForSystemsOverDedekindDomains} are vacuously fulfilled if $I = \emptyset$.

\begin{corollary}\label{NegativeResultForSystemsOverZ}
Let $a_1,\cdots,a_k,b_1,\cdots,b_k,c_1,\cdots,c_k \in \mathbb{Z}\setminus\{0\}$ and $n \in \mathbb{N}$ be such that the system of equations

\begin{equation}
    \begin{array}{ccccc}
         a_1x_1 & + & b_1y_1 & = & c_1w_1z_1^n\\
         a_2x_2 & + & b_2y_2 & = & c_2w_2z_2^n\\
         & & \vdots & & \\
         a_kx_k & + & b_ky_k & = & c_kw_kz_k^n
    \end{array}
\end{equation}
is partition regular over $\mathbb{Z}\setminus\{0\}$. Let

\begin{equation}
    I := \bigcap_{i = 1}^k \left\{\frac{a_i}{c_i},\frac{b_i}{c_i},\frac{a_i+b_i}{c_i}\right\}.
\end{equation}
\begin{enumerate}[(i)]
    \item If $4\nmid n$, then $I$ contains an $n$th power.
    
    \item If $4|n$, then $I$ contains an $\frac{n}{2}$th power.
\end{enumerate}
\end{corollary}

\begin{proof}
Let us assume for the sake of contradiction that one of items (i) and (ii) were false. Since $|I| \le 2$, we may invoke Lemma \ref{2VariableLemma} to find a prime $p$ satisfying the conditions of Theorem \ref{NegativeResultForSystemsOverDedekindDomains} to attain the desired contradiction.
\end{proof}

\begin{remark}
The following systems of equations are not partition regular as seen by an application of Corollary \ref{NegativeResultForSystemsOverZ}:
\begin{enumerate}[(i)]
    \item \[\begin{array}{rcrcc}
        3^4\cdot4^2\cdot5^2x_1 & + & 3^2\cdot4^4\cdot5^2y_1 & = & w_1z_1^4  \\
        5^4\cdot12^2\cdot13^2x_2 & + & 5^2\cdot12^4\cdot13^2y_2 & = & w_2z_2^4;
    \end{array}\]
    
    \item \[\begin{array}{rcrcc}
        16x_1 & + & 17y_1 & = & w_1z_1^8  \\
        33x_2 & + & (2^{12}-33)y_2 & = & w_2z_2^8;
    \end{array}\]
    
    \item \[\begin{array}{rcrcc}
        2^nx_1 & + & 3^ny_1 & = & w_1z_1^n \\
        3^nx_2 & + & 7^ny_2 & = & w_2z_2^n \\
        7^nx_3 & + & 2^ny_3 & = & w_3z_3^n,\\
        \text{for} & \text{any}& n \in \mathbb{N};
    \end{array}\]
    
    \item \[\begin{array}{rcrcl}
        9x_1 & + & 16y_1 & = & w_1z_1^2 \\
        25x_2 & - & 9y_2 & = & w_2z_2^2 \\
        25x_3 & - & 16y_3 & = & w_3z_3^2\\
        9x_4 & + & 7y_4 & = & w_4z_4^2.
    \end{array}\]
\end{enumerate}
A few remarks are in order regarding these examples. In Example (i) neither of the constituent equations of the system are individually partition regular. This fact can be determined through the use of Lemma \ref{MinimalRadoConditionApplied}, but not through the use of Theorem \ref{MainToolForNegativeResults} alone, despite the similarity of the proofs of Theorem \ref{MainToolForNegativeResults} and \ref{NegativeResultForSystemsOverDedekindDomains}. In Example (ii) we do not currently know whether either of the constituent equations of the system are individually partition regular (cf. Section \ref{Conclusion}). In Example (iii) any proper subsystem of equations is partition regular as a consequence of Theorem \ref{PositiveResultForSystems}. Theorem \ref{PositiveResultForSystems} also shows us that in Example (iv) the system of equations becomes partition regular if any equation other than the first equation is removed from the system.
\end{remark}
\section{Conjectures and Concluding Remarks}
\label{Conclusion}

Theorem \ref{MainResult} and Corollary \ref{SomeMoreExamples} naturally lead us to the following conjecture.

\begin{conjecture}	\label{MainConjecture}
	Given $a,b,c \in \mathbb{Z}\setminus\{0\}$ and $n \in \mathbb{N}$, the equation
	\begin{equation}
		ax+by = cwz^n
	\end{equation}
	is partition regular over $\mathbb{Z}\setminus\{0\}$ if and only if one of $\frac{a}{c},\frac{b}{c},\frac{a+b}{c}$ is an $n$th power in $\mathbb{Q}$. 
\end{conjecture}

We see that the situation in which we have yet to resolve Conjecture \ref{MainConjecture} fully is when $n$ is even and one of $\frac{a}{c},\frac{b}{c},$ or $\frac{a+b}{c}$ is an $\frac{n}{2}$th power in $\mathbb{Q}$. Since some special instances of this situation have been resolved in Corollary \ref{SomeMoreExamples}, we list here some equations whose partition regularity remains unknown. Firstly, the equation
\begin{equation}
	33x+(2^{12}-33)y = wz^8
\end{equation}
is not expected to be partition regular since $33, 2^{12}-33,$ and $2^{12}$ are not $8$th powers, but $2^{12}$ is an $8$th power modulo $p$ for every prime $p$, and $33$ is an $8$th power in $\mathbb{Z}_2$, so we are unable to apply Theorem \ref{MainToolForNegativeResults} or Theorem \ref{MinimalRadoConditionApplied}. Similarly, the equation
\begin{equation}
    16x+17y = wz^8
\end{equation}
is not expected to be partition regular since $16, 17,$ and $33$ are not $8$th powers, but $16$ is an $8$th power modulo $p$ for every prime p, and $16$ is also an $8$th power in $\mathbb{Z}_3$ and $\mathbb{Z}_{11}$, so we are once again unable to apply Theorem \ref{MainToolForNegativeResults} or Theorem \ref{MinimalRadoConditionApplied}. Next, we see that for any coprime $a,b \in \mathbb{N}$ for which $a,b,$ and $a+b$ are not squares, the equation
\begin{equation}
	a^2b(a+b)x+ab^2(a+b)y = wz^2
\end{equation}
is not expected to be partition regular since none of $a^2b(a+b), ab^2(a+b),$ and $a^2b(a+b)+ab^2(a+b) = ab(a+b)^2$ are squares. However, at least one of $a^2b(a+b), ab^2(a+b),$ and $ab(a+b)^2$ is a square modulo $p$ for any prime $p$, so we cannot make use of Theorem \ref{MainToolForNegativeResults}. We have seen in item (iii) of Corollary \ref{SomeMoreExamples} that Theorem \ref{MinimalRadoConditionApplied} can be used in some cases, but it is unclear to the authors as to whether or not Theorem \ref{MinimalRadoConditionApplied} can be used in all cases. A similar difficulty arises in the case of $n = 4$. Recalling that for any $m,n,k \in \mathbb{N}$ we have $(2mnk)^2+\left(km^2-kn^2\right)^2 = \left(km^2+kn^2\right)^2$, we take $k = 2mn(m^2-n^2)(m^2+n^2)$ and consider the equation
\begin{equation}\label{n=4Difficulty}
    \left(2mnk\right)^2x+\left(km^2-kn^2\right)^2y = wz^4.
\end{equation}
We see that for any prime $p$ at least one of $\alpha := 2mnk, \beta := km^2-kn^2,$ or $\gamma := km^2+kn^2$ will be a perfect square modulo $p$, and hence one of $\alpha^2, \beta^2,$ or $\gamma^2 = \alpha^2+\beta^2$ will be a fourth power modulo $p$, so Equation \eqref{n=4Difficulty} is not susceptible to Theorem \ref{MainToolForNegativeResults}. We saw in item (v) of Corollary \ref{SomeMoreExamples} that Theorem \ref{MinimalRadoConditionApplied} can occasionally be of use in this situation, but it is once again unclear to the authors whether or not Theorem \ref{MinimalRadoConditionApplied} can always be used in this situation.

While the methods of this paper are not strong enough to fully resolve Conjecture \ref{MainConjecture}, they are strong enough to prove the following Theorem.

\begin{theorem}	\label{UnprovenGeneralTheorem}
	Fix $a_1,\cdots,a_r,b_1,\cdots,b_s,c\in \mathbb{Z}\setminus\{0\}$.
	\begin{itemize}
		\item[(i)] If min$(b_1,\cdots,b_s) \ge 2$ then the equation
		\begin{equation}\label{GeneralizedMainEquation}
			\sum_{i = 1}^ra_ix_i = c\prod_{j = 1}^sy_j^{b_j}
		\end{equation}
		is partition regular over $\mathbb{Z}\setminus\{0\}$ if and only if there exists $\emptyset \neq F \subseteq [1,r]$ for which $\sum_{i \in F}a_i = 0$. 
		
		\item[(ii)] If $s \ge 2$, then the equation
		\begin{equation}\label{GeneralizedMainEquation2}
			\sum_{i = 1}^ra_ix_i = cy_1\prod_{j = 2}^sy_j^{b_j}
		\end{equation}
		is partition regular over $\mathbb{Z}\setminus\{0\}$ if for some $\emptyset \neq F \subseteq [1,r], s_F := \sum_{i \in F}\frac{a_i}{c}$ is an $n$th power in $\mathbb{Q}$, where $n = \sum_{j = 2}^sb_j$. Furthermore, if there exists $F \subseteq [1,r]$ for which $s_F$ is a $n$th power in $\mathbb{Q}_{\ge 0}$, then Equation \eqref{GeneralizedMainEquation2} is partition regular over $\mathbb{N}$.
		
		\item[(iii)] The equation
		\begin{equation}
			\sum_{i = 1}^ra_ix_i = cy_1\prod_{j = 2}^sy_j^{b_j}
		\end{equation}
		is not partition regular over $\mathbb{Q}\setminus\{0\}$ if there exists infinitely many primes $p$ such that for any $\emptyset \neq F \subseteq [1,r], \sum_{i \in F}\frac{a_i}{c}$ is not an $n$th power modulo $p$, where $n = \sum_{j = 2}^sb_j$.
	\end{itemize}
\end{theorem}

This naturally leads to the following generalization of Conjecture \ref{MainConjecture}.

\begin{conjecture}		\label{GeneralConjecture}
	Given $a_1,\cdots,a_r,c \in \mathbb{Z}\setminus\{0\}$ and $b_2,\cdots,b_s\in \mathbb{N}$, the equation
	\begin{equation}
		\sum_{i = 1}^ra_ix_i = cy_1\prod_{j = 2}^sy_j^{b_j}
	\end{equation}
	is partition regular over $\mathbb{Z}\setminus\{0\}$ if and only if there is some $\emptyset \neq F \subseteq [1,r]$ for which $\sum_{i \in F}\frac{a_i}{c}$ is an $n$th power in $\mathbb{Q}$, where $n = \sum_{j = 2}^sb_j$.

\end{conjecture}

We have already seen that the methods of this paper cannot be extended to prove Conjecture \ref{GeneralConjecture} even when $r = 2$. When $r > 2$, there are even more problematic cases to consider. For example, if $r = 4$ then at least one of $2, 5, 10,$ or $20$ is a perfect cube modulo $p$ for any prime $p$, so we are unable to use Theorem \ref{UnprovenGeneralTheorem}(iii) to show that the equation
\begin{equation}
	2x_1+5x_2+10x_3+20x_4 = y_1y_2^3
\end{equation}
is not partition regular over $\mathbb{Z}$. Furthermore, since $7^3 \equiv 10 \pmod{37}$, we may use Hensel's lemma to see that $10$ is a perfect cube in $\mathbb{Z}_{37}$, so analogues of Theorem \ref{MinimalRadoConditionApplied} will be of no use here.

In light of Theorem \ref{GeneralPositiveResultInDomains} and Corollary \ref{MainNegativeResultForNumberFields} it is natural to pose the following question.

\begin{question}\label{MostGeneralConjecture}
Given an integral domain $R$, $r_1,\cdots,r_k, c \in R\setminus\{0\}$, and $n_1,\cdots,n_s \in \mathbb{N}$, when is

\begin{equation}
    \sum_{i = 1}^kr_ix_i = c\prod_{j = 1}^sy_j^{b_j}
\end{equation}
partition regular over $R\setminus\{0\}$?
\end{question}

An analog of Theorem \ref{MaximalRadoCondition} would have to be proven for polynomial equations over $R$ instead of $\mathbb{N}$ in order to prove an analog of Theorem \ref{UnprovenGeneralTheorem}(i), which would help partially answer Question \ref{MostGeneralConjecture}. 
\begin{remark}
In light of Remark \ref{NvsZRemark} we are led to ask about sign obstructions to partition regularity of polynomial equations in rings of integers of totally real number fields. Let us consider for example the number field $K = \mathbb{Q}\left[\sqrt{2}\right]$, which is totally real since all of its embeddings into $\mathbb{C}$ turn out to be embeddings into $\mathbb{R}$. We recall that $\mathcal{O}_K = \mathbb{Z}\left[\sqrt{2}\right]$ is the ring of integers of $K$. It is clear that the equation

\begin{equation}\label{StillNotPR}
    x+y = -wz
\end{equation}
is not partition regular regular over $\mathbb{Z}\left[\sqrt{2}\right]_{> 0}$ due to sign obstructions despite being partition regular over $\mathbb{Z}\left[\sqrt{2}\right]\setminus\{0\}$ as a consequence of Theorem \ref{GeneralPositiveResultInDomains}. We are unable to determine whether or not the equation

\begin{equation}\label{PotentiallyPR}
    2x-2\sqrt{2}y = wz^3
\end{equation}
is partition regular over $\mathbb{Z}\left[\sqrt{2}\right]_{> 0}$ since there are no sign obstructions but we are unable to apply Theorem \ref{GeneralPositiveResultInDomains} and we are unable to apply the methods of Theorem \ref{GeneralPositiveResult} since $-2\sqrt{2} = (-\sqrt{2})^3$ and $-\sqrt{2} \notin \mathbb{Z}\left[\sqrt{2}\right]_{> 0}$. Unlike Remark \ref{NvsZRemark}, we may take this line of inquiry a step further by examining the semiring $P$ of totally positive elements of $\mathbb{Z}\left[\sqrt{2}\right]$, which are those elements that remain positive under every embedding of $\mathbb{Q}\left[\sqrt{2}\right]$ into $\mathbb{R}$. In this case we can directly determine $P$ to be given by $P = \{a+b\sqrt{2}\ |\ a > |b|\sqrt{2}\}$. It is clear that Equation \eqref{StillNotPR} is not partition regular over $P$ since $P \subseteq \mathbb{Z}\left[\sqrt{2}\right]_{> 0}$, but can we determine whether or not Equation \eqref{PotentiallyPR} is partition regular over $P$? Furthermore, it can be shown using the techniques of this paper that the equation

\begin{equation}\label{PartiallyPR}
    \sqrt{2}x+2\sqrt{2}y = wz^3
\end{equation}
is partition regular over $\mathbb{Z}\left[\sqrt{2}\right]_{> 0}$ since $2\sqrt{2} = (\sqrt{2})^3$ and $\sqrt{2} \in \mathbb{Z}\left[\sqrt{2}\right]_{> 0}$. However, $\sqrt{2}$ is a positive element of $\mathbb{Z}\left[\sqrt{2}\right]$ that is not totally positive, so Equation \eqref{PartiallyPR} is not partition regular over $P$ since $w,x,y,z \in P$ would result in the left hand side of the equation being positive but not totally positive, while the right hand side of the equation would be totally positive. 
\end{remark}
Now let us consider the equation

\begin{equation}\label{PartiallyPR2}
    2x+2\sqrt{2}y = wz^3.
\end{equation}
We can show that Equation \eqref{PartiallyPR2} is partition regular over $\mathbb{Z}\left[\sqrt{2}\right]_{> 0}$ using the considerations from above, but now that there are no ``generalized sign obstructions'' can we also determine whether or not Equation \eqref{PartiallyPR2} is partition regular over $P$?

In light of Theorem \ref{PositiveResultForSystems} and Corollary \ref{NegativeResultForSystemsOverZ} we are led to Conjecture \ref{FinalConjectureIntro} from the introduction, so we conclude with some examples of systems of equations whose partition regularity remains unknown. We are unable to apply Theorem \ref{NegativeResultForSystemsOverDedekindDomains} to the system of equations

\begin{equation}\label{UnresolvedSystem}
    \begin{array}{rcrcc}
         16x_1 & + & 17y_1 & = & w_1z_1^8  \\
         33x_2 & - & 17y_2 & = & w_2z_2^8 
    \end{array}
\end{equation}
since $I = \{16,33\}$ and $16$ is an $8$th power modulo every prime $p$. Since $33$ is an $8$th power in $\mathbb{Z}_2$ we also cannot expect a generalization of Theorem \ref{MinimalRadoConditionApplied} to systems of equations to help determine the partition regularity of the system in \eqref{UnresolvedSystem}. We are also unable to apply Theorem \ref{NegativeResultForSystemsOverDedekindDomains} to the system of equations

\begin{equation}
    \begin{array}{rcrcc}
         5^4x_1 & + & 3^6y_1 & = & w_1z_1^{12}  \\
         (5^4-3^6)x_2 & + & 3^6y_2 & = & w_2z_2^{12},
    \end{array}
\end{equation}
since $I = \left\{5^4,3^6\right\}$, $5$ is a cube modulo $p$ when $p \equiv 2 \pmod{3}$, and $-3$ is a square modulo $p$ when $p \equiv 1\pmod{3}$, so one of $5^4$ and $3^6$ will be a $12$th power modulo any prime $p$ (cf. Remark \ref{FLTRemark}(ii)(a)).

\section{The Existence of Special Ultrafilters}\label{Appendix}
In this section we will review some knowledge about $\beta S$, the space of ultrafilters over a set $S$, so that we can provide a thorough proof of Theorems \ref{SpecialUltrafilter} and \ref{SpecialUltrafilter6} as Theorems \ref{SpecialUltrafilter2} and \ref{SpecialUltrafilter5}, respectively. After proving Theorem \ref{SpecialUltrafilter5}, we give a brief discussion comparing the methods that we use to show that certain equations are partition regular with the methods used in \cite{BergelsonSurvey} and \cite{SumsEqualProducts}. As a result of this discussion (cf. Remark \ref{DiscussionOfDomains}), we decide to prove a generalization of Theorem \ref{SpecialUltrafilter2} as Corollary \ref{SpecialUltrafilter4}. For a more comprehensive study of ultrafilters and their usage in the study of semigroups the reader is referred to \cite{AlgebraInTheSCC}. 

Let us recall some notation before proceeding further. We let $\mathcal{P}(S)$ denote the collection of subsets of $S$ and $\mathcal{P}_f(S)$ denote the collection of non-empty finite subsets of $S$. If $(S,\diamond)$ is a semigroup, then for $s \in S$ and $A \subseteq S$ we define $sA = \{s\diamond a\ |\ a \in A\}$ and $s^{-1}A = \{x \in S\ |\ s\diamond x \in A\}$. 

\begin{manualdefinition}{3}\label{UltrafilterDefinition2}
    Let $S$ be a set. $p \subseteq \mathcal{P}(S)$ is an {\it ultrafilter} if it satisfies the following properties:
    \begin{enumerate}[(i)]
		\item The empty set is not a member of $p$, i.e., $\emptyset \notin p$,
		
		\item if $A \in p$ and $A \subseteq B$ then $B \in p$,
		
		\item if $A,B \in p$ then $A\cap B \in p$,
		
		\item for any $A \subseteq S$, either $A \in p$ or $A^c \in p$.
	\end{enumerate}
\end{manualdefinition}

We denote the space of all ultrafilters over $S$ by $\beta S$. It is often useful to think about $\beta S$ as the set of finitely additive $\{0,1\}$-valued measures on the set $S$. The topology of $\beta S$ is generated by the basis of open sets $\big\{\hat{A}\big\}_{A \subseteq S}$, where

\begin{equation}
    \hat{A} := \left\{p \in \beta S\ |\ A \in p\right\}.
\end{equation}
Since $(\hat{A})^c = \widehat{A^c}$ for all $A \subseteq S$, we see that each $\hat{A}$ is also a closed set, and it is a fact that $\{\hat{A}\}_{A \subseteq S}$ also generates the topology of $\beta S$ as a basis of closed sets. We note that for any $s \in S$, the collection of sets given by $e_s := \{A \subseteq S\ |\ s \in A\}$ is an ultrafilter over $S$. Let $e:S\rightarrow\beta S$ be given by $e(s) = e_s$ and observe that $e$ is an injective map that naturally embeds $S$ inside of $\beta S$ as a dense subset. Furthermore, when we endow $S$ with the discrete topology, which will always be the case for the rest of this section, $e$ is a homeomorphism onto its image. An ultrafilter $p$ is a principal ultrafilter if $p \in e(S)$, and a nonprincipal ultrafilter otherwise. Since we naturally identify each $s \in S$ with the principal ultrafilter $e_s$, it is common to write $s$ in place of $e_s$, and we will be using this convention for the rest of this section.

The following theorem is a universal property that can be used to characterize $\beta S$.

\begin{theorem}[{\cite[Theorem 3.27]{AlgebraInTheSCC}}]\label{UniversalProperty}
Let $S$ be an infinite set with the discrete topology and let $\beta S$ denote the space of ultrafilters over $S$. Given any compact space $Y$ and any function $f:S\rightarrow Y$, there exists a unique continuous function $\tilde{f}:\beta S\rightarrow Y$ such that $\tilde{f}|_S = f$.  
\end{theorem}

A careful and repeated application of Theorem \ref{UniversalProperty} also allows one to extend binary operations from $S\times S$ to $\beta S\times \beta S$, which is of great use when $S$ is a semigroup.

\begin{theorem}[{\cite[Theorem 4.1]{AlgebraInTheSCC}}]\label{ContinuousExtension}
    Let $S$ be a set and let $\diamond$ be a binary operation defined on $S$. There is a unique binary operation $*:\beta S\times\beta S\rightarrow\beta S$ satisfying the following three conditions.
    \begin{enumerate}[(a)]
        \item For every $s,t \in S$, $s*t = s\diamond t$.
        \item For each $q \in \beta S$, the function $\rho_q:\beta S\rightarrow \beta S$ is continuous, where $\rho_q(p) = p*q$.
        \item For each $s \in S$, the function $\lambda_s:\beta S\rightarrow\beta S$ is continuous, where $\lambda_s(q) = s*q$.
    \end{enumerate}
\end{theorem}

It is worth mentioning that our decision to make $\rho_q$ continuous for all $q \in \beta S$ and $\lambda_s$ continuous only for all $s \in S \subseteq \beta S$ is because we are choosing to use the same conventions used in \cite{AlgebraInTheSCC}. Elsewhere in the literature it is common to work with the extension $*$ of $\diamond$ that makes $\lambda_q$ continuous for all $q \in \beta S$ and $\rho_s$ continuous only for all $s \in S$.

\begin{theorem}[{\cite[Theorem 4.12]{AlgebraInTheSCC}}]\label{Semigroup}
    If $(S,\diamond)$ is a semigroup and $*:\beta S\times\beta S\rightarrow\beta S$ is the operation given by Theorem \ref{ContinuousExtension}, then $*$ is an associative operation. Furthermore, for $p,q \in \beta S$ the ultrafilter $p*q$ is given by
    
    \begin{equation}
        A \in p*q \text{ if and only if }\{s \in S\ |\ s^{-1}A \in q\} \in p.
    \end{equation}
\end{theorem}

It is customary to denote the operation $*$ that is produced by Theorem \ref{ContinuousExtension} by $\diamond$ so that $\diamond$ represents an operation on $S\times S$ as well as $\beta S\times\beta S$, and we shall adopt this practice. In light of Theorem \ref{Semigroup}, we see that if $(S,\diamond)$ is a semigroup, then $(\beta S,\diamond)$ is also a semigroup. Since $\beta S$ is a compact Hausdorff space (see \cite[Theorem 3.18]{AlgebraInTheSCC}), and for each $q \in \beta\mathbb{N}$ the map $\rho_q:\beta\mathbb{N}\rightarrow\beta\mathbb{N}$ given by $\rho_q(p) = p\diamond q$ is continuous, we see that $(\beta S,\diamond)$ is a compact right topological semigroup.\footnote{A {\bf compact right topological semigroup} is a semigroup $(S,\diamond)$ that is also a compact Hausdorff space for which each of the maps $\rho_s:S\rightarrow S$ given by $\rho_s(t) = t\diamond s$ are continuous. Note that other sources may use $\rho_s(t) = s\diamond t$ in their definition and that we base our definition off of \cite[Definition 2.1]{AlgebraInTheSCC}.} We apologize to the reader for our seemingly excessive emphasis on the operation $\diamond$ of our semigroup $(S,\diamond)$, but we choose to do this to avoid confusion later on when we work with rings $(R,+,\cdot)$ and are forced to consider the semigroups $(R,+)$ and $(R,\cdot)$ separately. We will now collect some results from the literature on semigroups.

\begin{definition}
Let $(S,\diamond)$ be a semigroup. 
\begin{enumerate}[(i)]
\item We say $e \in S$ is an {\it idempotent} if $e\diamond e = e$. We let $E(S,\diamond)$ denote the set of idempotents of $(S,\diamond)$. 

\item We say $\emptyset \neq L \subseteq S$ is a {\it left ideal} if for any $s \in S$ and $\ell \in L$ we have $s\diamond\ell \in L$. We say $\emptyset \neq R \subseteq S$ is a {\it right ideal} if for any $s \in S$ and $r \in R$ we have $r\diamond s \in R$. In general, $I \subseteq S$ is an {\it ideal} if it is a left ideal and a right ideal. 

\item We call $L \subseteq S$ a {\it minimal left ideal} if $L$ is a left ideal that does not properly contain any other left ideal. Similarly, {\it the smallest ideal of S}, if it exists, is an ideal $I$ that is contained in every other ideal of $S$.\footnote{Note that not every semigroup possesses a smallest ideal. In the semigroup $(\mathbb{N},+)$ each of the sets $I_n := \{m \in \mathbb{N}\ |\ m \ge n\}$ is an ideal. Since there is no set that is contained in every $I_n$, we see that $(\mathbb{N},+)$ does not have a smallest ideal.} If $(S,\diamond)$ does possess a smallest ideal, then we denote it by $K(S,\diamond)$ and observe that $K(S,\diamond)$ is also a semigroup.
\end{enumerate}
\end{definition}

\begin{theorem}[{\cite[Corollary 2.6]{AlgebraInTheSCC}}]
    Let $(S,\diamond)$ be a compact right topological semigroup. Then every left ideal of $S$ contains a minimal left ideal. Minimal left ideals are closed, and each minimal left ideal has an idempotent.
\end{theorem}

\begin{theorem}[{\cite[Theorem 2.8]{AlgebraInTheSCC}}]
    Let $(S,\diamond)$ be a semigroup. If $S$ has a minimal left ideal, then $K(S,\diamond)$ exists and $K(S,\diamond) = \bigcup\{L\ |\ L\text{ is a minimal left ideal of }S\}$.
\end{theorem}

\begin{lemma}[{\cite[Theorems 1.38(d) and 1.60]{AlgebraInTheSCC}}]\label{MinimalIdempotentsLemma}
Let $(S,\diamond)$ be a semigroup and let $\le$ be the partial ordering on the set of idempotents of $S$ given by $f \le e$ if and only if $fe = ef = f$. Assume that $S$ has a minimal left ideal that has an idempotent.
\begin{enumerate}[(i)]
    \item If $e \in S$ is an idempotent that is minimal with respect to $\le$, then $e$ is a member of some minimal left ideal of $S$. Such an idempotent is a {\it minimal idempotent}. It follows that $E(K(S,\diamond))$ is the set of minimal idempotents of $(S,\diamond)$.
    
    \item If $f \in S$ is an idempotent then there exists a minimal idempotent $e$ such that $e \le f$.
\end{enumerate}
\end{lemma}

\begin{definition}
    Let $(S,\diamond)$ be a semigroup and let $A \subseteq S$. Then $A$ is {\it central} if there exists an ultrafilter $p \in E(K(\beta S,\diamond))$ for which $A \in p$.
\end{definition}

\begin{theorem}[{\cite[Theorem 3.5]{PolynomialMPC}}]\label{GeneralDeuber}
    Let $(S,\diamond)$ be a commutative semigroup, let $\ell \in \mathbb{N}$, and let $A \subseteq \mathbb{N}$ be a central set. 
    \begin{enumerate}[(i)]
        \item There exists $b,g \in A$ such that 
    
    \begin{equation}
        b, b\diamond g, b\diamond g\diamond g, \cdots, b\diamond\underbrace{g\diamond g\diamond\cdots\diamond g}_{\ell} \in A.
    \end{equation}
    \item If $(S,\diamond)$ is a commutative group and $c:S\rightarrow S$ is a homomorphism for which $[S:c(S)] < \infty$, then there exists $b,g \in S$ such that
    
    \begin{equation}
        \{c(g)\}\cup\left\{c(b)\diamond g^j\right\}_{j = -\ell}^{\ell} \subseteq A.
    \end{equation}
    \end{enumerate}
\end{theorem}

\begin{lemma}[{\cite[Theorem 17.3]{AlgebraInTheSCC}}]\label{SpecialUltrafilter3}
    There exists an ultrafilter $p \in \beta\mathbb{N}$ such that every $A \in p$ is a central subset of $(\mathbb{N},+)$ and a central subset of $(\mathbb{N},\cdot)$.
\end{lemma}

We are now ready to prove Theorem \ref{SpecialUltrafilter2}, which implies Theorem \ref{SpecialUltrafilter}.

\begin{theorem}\label{SpecialUltrafilter2}
	Let $p \in \beta\mathbb{N}$ be such that every $A \in p$ is a central subset of $(\mathbb{N},+)$ and a central subset of $(\mathbb{N},\cdot)$. Then $p$ satisfies the following properties.
	\begin{enumerate}[(i)]
		\item For any $A \in p$ and $\ell \in \mathbb{N}$, there exist $b,g \in A$ with $\{bg^j\}_{j = 0}^{\ell} \subseteq A$.
		
		\item For any $A \in p$ and $h,\ell \in \mathbb{N}$, there exist $a,d \in \mathbb{N}$ for which $\{hd\}\cup\{ha+id\}_{i = -\ell}^{\ell} \subseteq A$.
		
		\item For any $s \in \mathbb{N}$, we have $s\mathbb{N} \in p$. 
	\end{enumerate}
\end{theorem}

\begin{proof}
     By Lemma \ref{SpecialUltrafilter3} let $p \in \beta\mathbb{N}$ be an ultrafilter such that every $A \in p$ is a central subset of $(\mathbb{N},+)$ and a central subset of $(\mathbb{N},\cdot)$. To see that $p$ satisfies (i), it suffices to observe that each $A \in p$ is a central subset of $(\mathbb{N},\cdot)$, so Theorem \ref{GeneralDeuber}(i) applies. To see that $p$ satisfies (ii) we first note that any central subset of $(\mathbb{N},+)$ is also a central subset of $(\mathbb{Z},+)$ by \cite[Theorem 2.9]{GeneralizedCentralSetsTheorem}, so we may apply Theorem \ref{GeneralDeuber}(ii) with the homomorphism $c$ given by $c(x) = hx$ for all $x \in \mathbb{Z}$ to find some $a,d \in \mathbb{Z}$ for which $\{hd\}\cup\{ha+id\}_{i = -\ell}^{\ell} \subseteq A$. Since $ha,hd \in A \subseteq \mathbb{N}$, we see that $a,d \in \mathbb{N}$. To see that $p$ satisfies $(iii)$, let $s \in \mathbb{N}$ be arbitrary and let us assume for the sake of contradiction that $s\mathbb{N} \notin p$. It follows that $(s\mathbb{N})^c \in p$, so by (ii) let $a,d \in \mathbb{N}$ be such that $\{sd\}\cup\{sa+id\}_{i = -1}^1 \subseteq (s\mathbb{N})^c$ to obtain the desired contradiction.
\end{proof}

\begin{remark}\label{DiscussionOfDomains}
In \cite{BergelsonSurvey} and \cite{SumsEqualProducts} the authors also used an ultrafilter $p$ for which every $A \in p$ is a central subset of $(\mathbb{N},+)$ and a central subset of $(\mathbb{N},\cdot)$ in order to show that the equation $x+y = wz$ is partition regular over $\mathbb{N}$, so it is unsurprising that we have managed to use such an ultrafilter to obtain our positive results over $\mathbb{N}$ and $\mathbb{Z}\setminus\{0\}$. Unfortunately, if $R$ is a general integral domain, then there may not exist an ultrafilter $p \in \beta R$ for which every $A \in p$ is a central subset of $(R,+)$ and a central subset of $(R\setminus\{0\},\cdot)$. Thankfully, we will see as a consequence of Theorem \ref{SpecialUltrafilter5} that we only need to work with central subsets of $(R\setminus\{0\},\cdot)$ in order to get our desired results for a general integral domain. In particular, the ultrafilter from Theorems \ref{SpecialUltrafilter6} and \ref{SpecialUltrafilter5} is just a minimal idempotent in $(\beta R\setminus\{0\},\cdot)$, so it is a corollary of Lemma \ref{PositiveResultForPowersInDomains} that any central subset of $(\mathbb{Z}\setminus\{0\},\cdot)$ (and consequently of $(\mathbb{N},\cdot)$) contains a solution to the equation $x+y = wz$. For the sake of completeness, we will still examine rings $R$ for which there exists an ultrafilter $p \in \beta R$ such that every $A \in p$ is a central subset of $(R,+)$ and a central subset of $(R\setminus\{0\},\cdot)$ after we prove Theorem \ref{SpecialUltrafilter5}. We remark that we do not require our rings to have a multiplicative unit. 

We would also like to point out to the reader that we work with central subsets of $(R\setminus\{0\},\cdot)$ instead of central subsets of $(R,\cdot)$ because $K(\beta R,\cdot) = \{0\}$, so $\{0\}$ is the only central subset of $(R,\cdot)$. For central subsets of $(R\setminus\{0\},\cdot)$ to be defined, we need $(R\setminus\{0\},\cdot)$ to be a semigroup, which is why we will only work with division rings and integral domains for the rest of this section. We also observe that the natural inclusion map $\iota:\beta(R\setminus\{0\})\rightarrow (\beta R)\setminus\{0\}$ given by 

\begin{equation}
    \iota(p) = \{A \subseteq R\ |\ A\setminus\{0\} \in p\} = p\cup\{A\cup\{0\}\ |\ A \in p\}
\end{equation}
is a homeomorphism. Furthermore, we see that for any $p_1,p_2 \in \beta(R\setminus\{0\})$ we have that $\iota(p_1\cdot p_2) = \iota(p_1)\cdot\iota(p_2)$, so $\iota$ is a semigroup isomorphism as well. Since $(\beta R)\setminus\{0\}$ and $\beta(R\setminus\{0\})$ are naturally isomorphic as compact right topological semigroups, we will write $\beta R\setminus\{0\}$ for $(\beta R)\setminus\{0\}$ without any worry for the potential confusion with $\beta(R\setminus\{0\})$.
\end{remark}

\begin{definition}
Let $R$ be a division ring. A family $\mathcal{F} \subseteq \mathcal{P}(R)$ is \textit{partition regular} if for every finite partition $R = \bigcup_{i = 1}^rC_i$, there exists $1 \le i_0 \le r$ and $F \in \mathcal{F}$ for which $F \subseteq C_{i_0}$. $\mathcal{F}$ is \textit{multiplicatively shift invariant} if for every $F \in \mathcal{F}$ and $r \in R\setminus\{0\}$ we have $rF \in \mathcal{F}$.
\end{definition}

\begin{lemma}\label{RadoIdealSetUp}
Let $R$ be an infinite integral domain and let $\mathcal{F} \subseteq \mathcal{P}_f(R\setminus\{0\})$ be a partition regular family that is also multiplicatively shift invariant. Let $I_{\mathcal{F}} \subseteq \beta R\setminus\{0\}$ denote the collection of ultrafilters $p$ such that for every $A \in p$ there exists $F \in \mathcal{F}$ for which $F \subseteq A$. $I_{\mathcal{F}}$ is an ideal of $(\beta R\setminus\{0\},\cdot)$. 
\end{lemma}

\begin{proof}
First let us show that $I_{\mathcal{F}}$ is nonempty. To this end, let us assume for the sake of contradiction that for each $p \in \beta R\setminus\{0\}$ there exists $A_p \in p$ such that there is no $F \in \mathcal{F}$ with $F \subseteq A$. Since $\big\{\widehat{A_p}\big\}_{p \in \beta R\setminus\{0\}}$ is an open cover of the compact space $\beta R\setminus\{0\}$, let $\big\{\widehat{A_{p_i}}\big\}_{i = 1}^r$ be a finite subcover. The desired contradiction follows from the observation that $R\setminus\{0\} = \bigcup_{i = 1}^rA_{p_i}$ is a partition in which no cell contains a member of $\mathcal{F}$. Now let us show that $I_{\mathcal{F}}$ is a left ideal. To this end, let $p \in I_{\mathcal{F}}$ and $q \in \beta R\setminus\{0\}$ both be arbitrary. We see that for $A \in q\cdot p$ we have

\begin{equation}
    \left\{r \in R\setminus\{0\}\ |\ r^{-1}A \in p\right\} \in q,
\end{equation}
so let $r \in R\setminus\{0\}$ be such that $r^{-1}A \in p$. Since $r^{-1}A \in p$, there is some $F \in \mathcal{F}$ for which $F \subseteq r^{-1}A$, so $rF \subseteq A$. Since $rF \in \mathcal{F}$ and $A$ was arbitrary, we see that $q\cdot p \in I_{\mathcal{F}}$, so $I_{\mathcal{F}}$ is a left ideal. Now let us show that $I_{\mathcal{F}}$ is a right ideal. To this end, let $p \in I_{\mathcal{F}}$ and $q \in \beta R\setminus\{0\}$ both be arbitrary.  We see that for $A \in p\cdot q$ we have

\begin{equation}
    \left\{r \in R\setminus\{0\}\ |\ r^{-1}A \in q\right\} \in p,
\end{equation}
so let $F = \{x_1,\cdots,x_k\} \in \mathcal{F}$ be such that $F \subseteq \left\{r \in R\setminus\{0\}\ |\ r^{-1}A \in q\right\}$. Since $\bigcap_{i = 1}^kx_i^{-1}A \in q$, let $y \in \bigcap_{i = 1}^kx_i^{-1}A$ be arbitrary and note that $yF \subseteq A$. Since $yF \in \mathcal{F}$ and $A$ was arbitrary, we see that $p\cdot q \in I_{\mathcal{F}}$, so $I_{\mathcal{F}}$ is also a right ideal.
\end{proof}

\begin{remark}
An example of a partition regular $\mathcal{F} \subseteq \mathcal{P}_f(R)$ that is also multiplicatively shift invariant is given by $\mathcal{F}_1 = \{\{a,b,a+b\}\}_{a,b \in R}$. In general, if ${\bf M} \in M_{m,n}(R)$ is a matrix for which the equation ${\bf M}\vec{x} = \vec{0}$ is partition regular, then we obtain another such example by taking $\mathcal{F}_{\bf M}$, the family of finite sets each of which contains a solution to the equation ${\bf M}\vec{x} = \vec{0}$.
\end{remark}

\begin{theorem}[\cite{RadoForRingsAndModules}]
Let $R$ be an infinite integral domain and let ${\bf M}$ be a matrix with entries in $R$. Then the system ${\bf M}\vec{x} = \vec{0}$ is partition regular over $R\setminus\{0\}$ if and only if ${\bf M}$ satisfies the columns condition (cf. Definition \ref{ColumnsCondition}).\footnote{In \cite{RadoForRingsAndModules} the statements of the results discuss partition regularity over $R$, not $R\setminus\{0\}$, but the definition of partition regularity that is used in \cite{RadoForRingsAndModules} explicitly forbids trivial solutions, which is why we may modify their statement to mirror our previous statements.}
\end{theorem}

\begin{lemma}\label{rRIdeal}
Let $R$ be an infinite integral domain. Let $H \subseteq \beta R\setminus\{0\}$ denote the collection of ultrafilters $p$ such that for every $r \in R\setminus\{0\}$ we have $rR \in p$. Then $H$ is an ideal of $(\beta R\setminus\{0\},\cdot)$.
\end{lemma}

\begin{proof}
Let $p \in H, q \in \beta R\setminus\{0\},$ and $r \in R\setminus\{0\}$ all be arbitrary. Let us first show that $H$ is a left ideal. We note that for any $s \in R$, we have $s^{-1}(rR) \supseteq rR$. Since $rR \in p$, we see that
\begin{equation}
    R \subseteq \left\{s \in R\setminus\{0\}\ |\ s^{-1}(rR) \in p\right\},\text{ hence}\left\{s \in R\setminus\{0\}\ |\ s^{-1}(rR) \in p\right\} \in q.
\end{equation}
It follows that $rR \in q\cdot p$, so $q\cdot p \in H$ and $H$ is indeed a left ideal. Now let us show that $H$ is a right ideal. We note that for any $s \in rR$ we have $s^{-1}(rR) \supseteq R$. Since $R \in q$, we see that
\begin{equation}
    rR \subseteq \left\{s \in R\setminus\{0\}\ |\ s^{-1}(rR) \in q\right\},\text{ hence}\left\{s \in R\setminus\{0\}\ |\ s^{-1}(rR) \in q\right\} \in p.
\end{equation}
It follows that $rR \in p\cdot q$, so $p\cdot q \in H$.
\end{proof}

\begin{theorem}[cf. Theorem \ref{SpecialUltrafilter6}]\label{SpecialUltrafilter5}
Let $R$ be an infinite integral domain. If $p \in K(\beta R\setminus\{0\},\cdot)$, then $p$ satisfies the following properties.
\begin{enumerate}[(i)]
   \item For any $A \in p$ and any partition regular family $\mathcal{F} \subseteq \mathcal{P}_f(R)$ that is also multiplicatively shift invariant, there exists $F \in \mathcal{F}$ with $F \subseteq A$.
    
    \item For every $r \in R\setminus\{0\}$, we have $rR \in p$.

    \item For any $A \in p, \ell \in \mathbb{N}$, and $h,j_1,j_2,\cdots,j_{\ell} \in R\setminus\{0\}$, there exists $a,d \in R$ for which $\{hd,ha\}\cup\{ha+j_id\}_{i = 0}^{\ell} \subseteq A$.

    \item Suppose that $p\cdot p = p$, i.e., $p \in E(K(\beta R\setminus\{0\},\cdot))$. For any $A \in p$ and $\ell \in \mathbb{N}$, there exists $b,g \in A$ with $\left\{bg^j\right\}_{j = 0}^{\ell} \subseteq A$.
\end{enumerate}
\end{theorem}

\begin{proof}
  Let $\mathscr{F} = \{\mathcal{F} \subseteq \mathcal{P}_f(R)\ |\ \mathcal{F}\text{ is partition regular and multiplicatively shift }\allowbreak\text{invariant}\}$, let $I_{\mathcal{F}}$ be as in Lemma \ref{RadoIdealSetUp}, and let $I = \bigcap_{\mathcal{F} \in \mathscr{F}}I_{\mathcal{F}}$. Since $I$ and $H$ (from Lemma \ref{rRIdeal}) are ideals of $(\beta R\setminus\{0\},\cdot)$, $I\cap H$ is also an ideal. Hence $K(\beta R\setminus\{0\},\cdot) \subseteq I\cap H$. Since $p \in I$, we see that condition (i) is satisfied. Since $p \in H$, we see that condition (ii) is satisfied. The fact that $p$ satisfies (iii) is a corollary of the fact that $p$ satisfies (i). We give a proof of this implication for the sake of completeness, since property (iii) is used in the earlier sections of this paper.

  To see that $p$ satisfies condition (iii), we first consider the system of equations

\begin{equation}\label{SystemForBrauer}
    \begin{array}{lcrcrcl}
         hx_3 & - & hx_2 & - & j_1x_1 & = & 0\\
         hx_4 & - & hx_2 & - & j_2x_1 & = & 0 \\
         & & & \vdots & & & \\
         hx_{\ell+2} & - & hx_2 & - & j_{\ell}x_1 & = & 0\\
         hx_{\ell+3} & - & x_2 & - & hx_{\ell+5} & = & 0\\
         hx_{\ell+4} & - & hx_2 & - & x_1 & = & 0.
    \end{array}
\end{equation}
Let us assume that $\{x_i\}_{i = 1}^{\ell+5} \subseteq R\setminus\{0\}$ is a solution to the system of equations in \eqref{SystemForBrauer}. Since $x_1 = h(x_{\ell+4}-x_2)$ and $x_2 = h(x_{\ell+3}-x_{\ell+5})$, we may write $x_1 = hd$ and $x_2 = ha$ for some $a,d \in R\setminus\{0\}$. It follows that $x_i = ha+j_id$ for $3 \le i \le \ell+2$, so it suffices to show that each $A \in p$ contains a solution to the system of equations in \eqref{SystemForBrauer}. To this end, we observe that $\mathcal{F} := \left\{\{x_i\}_{i = 1}^{\ell+5}\ |\ x_1,\cdots,x_{\ell+5}\text{ satisfy equation }\eqref{SystemForBrauer}\right\}$ is multiplicatively shift invariant, so we only need to show that $\mathcal{F}$ is also partition regular in order to apply (i). Let ${\bf M} \in M_{\ell+2,\ell+5}(R)$ be the matrix such that the equation ${\bf M}\vec{x} = \vec{0}$ represents the system of equations in \eqref{SystemForBrauer}, and we will proceed to show that ${\bf M}$ satisfies the columns condition. Let $\left\{\vec{c}_i\right\}_{i = 1}^{\ell+5}$ denote the columns of ${\bf M}$, with $\vec{c}_i$ representing the column corresponding to $x_i$. We see that 

\begin{equation}
    \vec{c}_1 = \begin{pmatrix}-j_1 \\ -j_2 \\ \vdots \\ -j_{\ell} \\ 0 \\ -1\end{pmatrix}, \vec{c}_2 = \begin{pmatrix}-h \\ -h \\ {\small \vdots} \\ -h \\ -1 \\ -h\end{pmatrix}, \vec{c}_i = \begin{pmatrix}0 \\ \vdots \\ 0 \\ h \\ 0 \\ \vdots \\ 0\end{pmatrix}\text{ for }3 \le i \le \ell+4,\text{ and }\vec{c}_{\ell+5} = \begin{pmatrix} 0 \\ 0 \\ \vdots \\ 0 \\ -h \\ 0\end{pmatrix}
\end{equation}
with the $h$ in $\vec{c}_i$ occuring in row $i-2$ for 
$3 \le i \le \ell+4$. Consider the partition of column indices $\{C_1,C_2,C_3\}$ given by $C_1 = \{\ell+3,\ell+5\}, C_2 = \{2,3,\cdots,\ell+2,\ell+4\},$ and $C_3 = \{1\}$. We see that

\begin{equation}
    \begin{array}{ccl}
        \vec{s}_1 & = & \vec{c}_{\ell+3}+\vec{c}_{\ell+5} = \vec{0},\\
        \vec{s}_2 & = & \vec{c}_{\ell+4}+\sum_{i = 2}^{\ell+2}\vec{c}_i = \frac{1}{h}\vec{c}_{\ell+5},\text{ and }\\
        \vec{s}_3 & = & \vec{c}_1 = \left(\sum_{i = 3}^{\ell+2}-\frac{j_{i-2}}{h}\vec{c}_i\right)-\frac{1}{h}\vec{c}_{\ell+4},
    \end{array}
\end{equation}
so ${\bf M}$ does indeed satisfy the columns condition.

  To see that $p$ satisfies (iv), it suffices to observe that each $A \in p$ is a central subset of $(R\setminus\{0\},\cdot)$, so the desired result follows from Theorem \ref{GeneralDeuber}(i).
\end{proof}

To conclude this section we will show that an integral domain $R$ possesses an ultrafilter $p \in \beta R$ such that every $A \in p$ is a central subset of $(R,+)$ and a central subset of $(R,\cdot)$ if and only if $R$ is homomorphically finite. 

\begin{definition}
     Let $(S,\diamond)$ be a semigroup and let $A \subseteq S$. $A$ is {\it syndetic} if and only if there exists some $G \in \mathcal{P}_f(S)$ such that $S = \bigcup_{t \in G}t^{-1}A$.
\end{definition}

We observe that if $(S,\diamond)$ is a group and $H \subseteq S$ is a subgroup, then $H$ is a syndetic subset of $S$ if and only if $[S:H] < \infty$.

\begin{definition}\label{HomFinDef}
A ring $R$ is a {\it  right (left) homomorphically finite} if for every $r \in R\setminus\{0\}$ the right (left) ideal $rR$ ($Rr$) is a finite index subgroup of $(R,+)$. $R$ is a {\it homomorphically finite} if for every $r \in R\setminus\{0\}$ the two-sided ideal generated by $r$ is a finite index subgroup of $(R,+)$.
\end{definition}

\begin{theorem}[{\cite[Corollary 1.3.3]{SkewFields}}]\label{OreCondition}
Let $R$ be a ring with no zero divisors\footnote{The reader is warned that in \cite{SkewFields} an integral domain is a not necessarily commutative ring with no zero divisors. Similarly, in \cite{SkewFields} a field refers to a not necessarily commutative division ring.} such that

\begin{equation}
    aR\cap bR \neq \{0\}\ \text{for all }\ a,b \in R\setminus\{0\}.
\end{equation}
Then the localization of $R$ at $R\setminus\{0\}$ is a division ring $D$ and the natural homomorphism $\lambda:R\rightarrow D$ is an embedding.
\end{theorem}

\begin{theorem}[{\cite[Theorem 5.8]{AlgebraInTheSCC}}]\label{IdempotentsHaveIPSets} Let $(S,\diamond)$ be a semigroup, let $p$ be an idempotent in $\beta S$, and let $A \in p$. There is a sequence $(x_n)_{n = 1}^{\infty}$ in $S$ such that $\left\{\prod_{n \in F}x_n\right\}_{F \in \mathcal{P}_f(\mathbb{N})} \allowbreak\subseteq A$.
\end{theorem}

\begin{theorem}\label{FiniteIndexIsIP*}
Let $(G,\diamond)$ be a group and let $H$ be a finite index subgroup of $G$. If $p \in \beta G$ is an idempotent, then $H \in p$.
\end{theorem}

\begin{proof}
Let $M = [G:H]$ and let us assume for the sake of contradiction that $H^c \in p$. By Theorem \ref{IdempotentsHaveIPSets} let $(x_n)_{n = 1}^{\infty}$ be a sequence in $G$ such that $\left\{\prod_{n \in F}x_n\right\}_{F \in \mathcal{P}_f(\mathbb{N})} \subseteq H^c$. Since $H^c$ is a disjoint union of $M-1$ cosets of $H$, let $1 \le j < k \le M$ be such that $\left(\prod_{i = 1}^jx_i\right)H = \left(\prod_{i = 1}^kx_i\right)H$. It follows that $\left(\prod_{i = j+1}^kx_i\right)H = H$, and hence $\prod_{i = j+1}^kx_i \in H$, which yields the desired contradiction.
\end{proof}

\begin{lemma}\label{MultiplicationPreservesMinimalIdempotents}
If $R$ is an infinite right (respectively left) homomorphically finite ring that has no zero divisors and $p \in E(K(\beta R,+))$, then for any $r \in R\setminus\{0\}$ we have $r\cdot p \in E(K(\beta R,+))$ (respectively $p\cdot r \in E(K(\beta R,+))$).
\end{lemma}

\begin{proof}
We only prove the desired result for $r\cdot p$ since the proof of the result for $p\cdot r$ is similar. Firstly, we would like to show that $R$ is a subring of a division ring $D$, so it suffices to show that $R$ satisfies the conditions of Theorem \ref{OreCondition}. Let $r,s \in R\setminus\{0\}$ be arbitrary and note that within the group $(R,+)$ we have $[R:rR\cap sR] \le [R:rR][R:sR] < \infty$. Since $R$ is infinite and has no zero divisors, we see that $|rR\cap sR| = \infty$, and hence $R$ embeds in some division ring $D$. 

Since $R$ is a ring, we see that for any $r \in R$ the map $\ell_r:R\rightarrow R$ given by $\ell_r(s) = rs$ is an endomorphism of the group $(R,+)$, and hence its unique continuous extension $\tilde{\ell_r}:\beta R\rightarrow \beta R$ is also an additive endomorphism by \cite[Lemma 2.14]{AlgebraInTheSCC}, so $r\cdot p = \ell_r(p)$ is an additive idempotent. It only remains to show that $r\cdot p$ is minimal. To this end, by Lemma \ref{MinimalIdempotentsLemma}(ii) let $q \in \beta R$ be a minimal idempotent for which $q \le r\cdot p$. We note that $\tilde{\ell_{r^{-1}}}:\beta D\rightarrow \beta D$ is also an additive endomorphism, and hence $r^{-1}\cdot q \le r^{-1}\cdot r\cdot p = p$. By Lemma \ref{FiniteIndexIsIP*} we see that $rR \in q$, so $R \in r^{-1}\cdot q$, and hence $r^{-1}\cdot q \in \beta R \subseteq \beta D$. Since $p \in \beta R$ is minimal and $r^{-1}\cdot q \le p$, we see that $r^{-1}\cdot q = p$, and hence $q = r\cdot p$.
\end{proof}

\begin{theorem}\label{ExistinceOfAdditiveAndMultiplicativeUltrafilters}
If $R$ is an infinite right homomorphically finite ring that has no zero divisors, then there exists an ultrafilter $p \in E(K(\beta(R)\setminus\{0\},\cdot))\cap \overline{E(K(\beta R,+))}$.
\end{theorem}

\begin{proof}
Using Lemma \ref{MultiplicationPreservesMinimalIdempotents} and the continuity of right multiplication we see that

\begin{alignat}{2}
    &(R\setminus\{0\})\cdot E(K(\beta R,+)) \subseteq E(K(\beta R,+)),\text{ hence}\numberthis\\
    &\overline{(R\setminus\{0\})}\cdot E(K(\beta R,+)) \subseteq \overline{E(K(\beta R,+))},\text{ so}\\
    & (\beta(R)\setminus\{0\})\cdot E(K(\beta R,+)) \subseteq \overline{E(K(\beta R,+))}.
\end{alignat}

Since $(\beta(R)\setminus\{0\})\cdot E(K(\beta R,+))$ is a left ideal of $(\beta R\setminus\{0\},\cdot)$, it contains a minimal idempotent, which is the desired $p$.
\end{proof}

\begin{corollary}\label{SpecialUltrafilter4}
Let $R$ be an infinite right homomorphically finite ring that has no zero divisors. There exists an ultrafilter $p \in \beta R$ such that every $A \in p$ is a central subset of $(R,+)$ and a central subset of $(R\setminus\{0\},\cdot)$.
\end{corollary}

\begin{proof}
Using Theorem \ref{ExistinceOfAdditiveAndMultiplicativeUltrafilters} let us pick some $p \in E(K(\beta R\setminus\{0\},\cdot))\cap\overline{E(K(\beta R,+))}$ and let $A \in p$ be arbitrary. Since $p \in E(K(\beta R\setminus\{0\},\cdot)) = E(K(\beta(R\setminus\{0\}),\cdot))$, we see that $A$ is a central subset of $(R\setminus\{0\},\cdot)$. Since $p \in \overline{E(K(\beta R,+))}$ and $\hat{A}$ is an open neighborhood of $p$, pick some $q \in E(K(\beta R,+))\cap\hat{A}$. Since $A \in q$ and $q \in E(K(\beta R,+))$, we see that $A$ is a central subset of $(R,+)$.
\end{proof}

\begin{remark}
To see why we need to assume that $R$ is a right homomorphically finite ring let us consider the integral domain $\mathbb{Q}[x]$. It is clear that $x\mathbb{Q}[x]$ is not a finite index subgroup of $(\mathbb{Q}[x],+)$, and we will see in Theorem \ref{NecessityOfHomFin} that $\widehat{x\mathbb{Q}[x]} \supseteq \overline{K(\beta\mathbb{Q}[x],\cdot)}$ while $\widehat{x\mathbb{Q}[x]}\cap\overline{K(\beta\mathbb{Q}[x],+)} = \emptyset$. It follows that $\widehat{x\mathbb{Q}[x]} \supseteq \overline{E(K(\beta\mathbb{Q}[x],\cdot))}$ while $\widehat{x\mathbb{Q}[x]}\cap\overline{E(K(\beta\mathbb{Q}[x],+))} = \emptyset$, so $x\mathbb{Q}[x]$ intersects every central subset of $(\mathbb{Q}[x]\setminus\{0\},\cdot)$ even though it is not a central subset of $(\mathbb{Q}[x],+)$. Furthermore, in the proof of Lemma \ref{ExistinceOfAdditiveAndMultiplicativeUltrafilters} we used the continuity of right multiplication in $\beta(R\setminus\{0\})$ which is why we had to assume that the ring $R$ was a right homomorphically finite ring. The same proof yields an analogous result for left homomorphically finite rings if you work with the extension of $\cdot$ from $R\setminus\{0\}$ to $\beta(R\setminus\{0\})$ that makes left multiplication continuous. Note that the minimal idempotents of $\beta(R\setminus\{0\},\cdot)$ may change depending on which extension of $\cdot$ from $R\setminus\{0\}$ to $\beta(R\setminus\{0\})$ you use as shown in \cite[Theorem 13.40.2]{AlgebraInTheSCC}.
\end{remark}

\begin{theorem}\label{NecessityOfHomFin}
Let $R$ be an infinite integral domain that is not homomorphically finite and let $r_0 \in R\setminus\{0\}$ be such that $[R:r_0R] = \infty$. We have

\begin{equation}
    \widehat{r_0R} \supseteq \overline{K(\beta R\setminus\{0\},\cdot)}\text{ and }\widehat{r_0R}\cap\overline{K(\beta R,+)} = \emptyset.
\end{equation}
In particular, we have

\begin{equation}
    \overline{K(\beta R\setminus\{0\},\cdot)}\cap\overline{K(\beta R,+)} = \emptyset.
\end{equation}
\end{theorem}

\begin{proof}
Letting $H$ be as in Lemma \ref{rRIdeal} and observe that

\begin{equation}
    H = \bigcap_{r \in R\setminus\{0\}}\widehat{rR},
\end{equation}
so $H$ is a closed. Since $H$ is an ideal by Lemma \ref{rRIdeal}, we see that $K(\beta R\setminus\{0\},\cdot) \subseteq H$, and hence

\begin{equation}
    \overline{K(\beta R\setminus\{0\},\cdot)} \subseteq \overline{H} = H \subseteq \widehat{r_0R}.
\end{equation}
Now let us assume for the sake of contradiction that there exists some $p \in \widehat{r_0R}\cap\overline{K(\beta R,+)}$. Since $p \in \overline{K(\beta R,+)}$ and $\widehat{r_0R}$ is an open neighborhood of $p$, pick some $q \in \widehat{r_0R}\cap K(\beta R,+)$. Using \cite[Theorem 4.39]{AlgebraInTheSCC} we see that $\{r \in R\ |\ -r+r_0R \in q\}$ is a syndetic subset of $(R,+)$. Noting that for each $r \in R$ we have $-r+r_0R = r_0R$ if $r \in r_0R$ and $(-r+r_0R)\cap r_0R = \emptyset$ if $r \notin r_0R$, we see that $\{r \in R\ |\ -r+r_0R \in q\} = r_0R$. Since $[R:r_0R] = \infty$, we see that $r_0R$ is not a syndetic subset of $(R,+)$, which yields the desired contradiction.
\end{proof}

\begin{corollary}
Let $R$ be an infinite integral domain that is not homomorphically finite. There does not exist an ultrafilter $p \in \beta R$ such that every $A \in p$ is a central subset of $(R,+)$ and a central subset of $(R\setminus\{0\},\cdot)$.
\end{corollary}

\begin{proof}
Let us assume for the sake of contradiction that such an ultrafilter $p \in \beta R$ did exist. Let $A \in p$ be arbitrary. Since $A$ is a central subset of $(R\setminus\{0\},\cdot)$ let $q \in E(K(\beta R\setminus\{0\},\cdot))$ be such that $A \in q$. Since $A$ is a central subset of $(R,+)$, let $q' \in E(K(\beta R,+))$ be such that $A \in q'$. We see that $\hat{A}$ is an open neighborhood of $p$ that contains $q$ and $q'$, and hence

\begin{equation}
p \in \overline{E(K(\beta R\setminus\{0\},\cdot))}\cap\overline{E(K(\beta R,+))},
\end{equation}
which contradicts Theorem \ref{NecessityOfHomFin}.
\end{proof}

\smallskip

\noindent {\bf Acknowledgements.} We would like to thank Vitaly Bergelson for proposing this topic of research and for his helpful guidance. We would also like to thank the referee for his/her useful comments during his/her careful review of this paper.


\bibliographystyle{integers.bst}

\end{document}